\documentclass[11pt,reqno]{amsart}
\usepackage{amsaddr}
\usepackage[pdftex]{graphicx}
\usepackage{pgfplots}
\usepackage{epstopdf}
\usepackage{geometry}
\usepackage{mathrsfs}
\usepackage{stix}
\usepackage{amsmath,amsthm}
\usepackage{array,booktabs}
\usepackage{cite}
\usepackage{hyperref}
\usepackage{caption}
\usepackage[symbol]{footmisc}
\usepackage{subcaption}
\usepackage[normalem]{ulem}
\hypersetup{
    colorlinks=true, 
    linktoc=all,     
    allcolors=blue,  
}
\usepackage{mathrsfs}
\usetikzlibrary{arrows.meta}
\usepackage{pst-func}
\usepackage{url}
\DeclareMathOperator{\arcsinh}{arcsinh}
\usepackage[]{mdframed}
\newtheorem{theorem}{Theorem}[section]
\newtheorem{lemma}{Lemma}[section]
\newtheorem{proposition}{Proposition}[section]
\newtheorem{corollary}{Corollary}[section]

\theoremstyle{definition}
\newtheorem{definition}{Definition}[section]
\newtheorem{example}{Example}[section]
\newtheorem{remark}{Remark}[section]
\numberwithin{equation}{section}
\numberwithin{equation}{subsection}
\usepackage{wasysym}

\begin{document}
\pagenumbering{arabic}
\pgfplotsset{compat=1.18}

\title[Poncelet pairs of a circle and parabolas from a confocal family ]{Poncelet pairs of a circle and parabolas from a confocal family and Painlev\'e VI equations}
\author{Vladimir Dragovi\'{c}$^{1,2}$ AND Mohammad Hassan Murad$^{1,3}$}
\address{$^{1}$Department of Mathematical Sciences,
\\The University of Texas at Dallas,\\
800 W Campbell Rd, Richardson, Texas 75080\\
$^{2}$Mathematical Institute SANU, Belgrade, Serbia\\
$^{3}$Department of Mathematics and Natural Sciences\\
BRAC University, Dhaka, Bangladesh}
\email{vladimir.dragovic@utdallas.edu; mohammadhassan.murad@utdallas.edu}

\begin{abstract}
 We study pairs of conics $(\mathcal{D},\mathcal{P})$, called  \textit{$n$-Poncelet pairs},  such that an $n$-gon,  called an \textit{$n$-Poncelet polygon}, can be inscribed into $\mathcal{D}$ and circumscribed about $\mathcal{P}$. Here $\mathcal{D}$ is a circle and $\mathcal{P}$ is a parabola from a confocal pencil $\mathcal{F}$ with the focus $F$. We prove that the circle contains $F$ if and only if every parabola $\mathcal{P}\in\mathcal{F}$ forms a $3$-Poncelet pair with the circle. We prove that the center of $\mathcal{D}$ coincides with  $F$ if and only if every parabola $\mathcal{P}\in \mathcal{F}$ forms a $4$-Poncelet pair with the circle. We refer to such property, observed for $n=3$ and $n=4$, as \textit{$n$-isoperiodicity}. We prove that $\mathcal{F}$ is not $n$-isoperiodic with any circle $\mathcal{D}$ for $n$ different from $3$ and $4$. Using isoperiodicity, we construct explicit algebraic solutions to Painlev\'e VI equations.
\end{abstract}


\maketitle{}

{\it Dedicated to the Memory of Academician A. A. Bolibrukh, on the occasion of his 75th anniversary}

\section{Introduction}\label{sec.1}
Let two smooth conics $\mathcal{C}_1$, $\mathcal{C}_2$ be given. Suppose there is an $n$-gon inscribed in $\mathcal{C}_1$ and circumscribed about $\mathcal{C}_2$. Then, by \textit{Poncelet's Theorem} (see e.g. \cite{Flatto2009,Dragovic-Radnovic2011}), for any point of $\mathcal{C}_1$, there exists such an $n$-gon, inscribed in $\mathcal{C}_1$ and circumscribed about $\mathcal{C}_2$, having this point as one of its vertices. See Figure \ref{fig.1} for an illustration of Poncelet's Theorem.

\begin{definition}\label{def.1.1}
For a natural number $n \geq 3$, an $n$-gon inscribed in $\mathcal{C}_1$ and circumscribed about $\mathcal{C}_2$ is called an \textit{$n$-Poncelet polygon}  and the pair $(\mathcal{C}_1,\mathcal{C}_2)$ is called an \textit{$n$-Poncelet pair} if there exists an $n$-Poncelet polygon inscribed in $\mathcal{C}_1$ and circumscribed about $\mathcal{C}_2$.
\end{definition}

In this paper, we study $n$-Poncelet pairs of a circle and a parabola from a confocal pencil. We will denote by $\mathcal{D}(E,R)$  the circle centered at $E$ with radius $R$. By $\mathcal{P}=\mathcal{P}(p)$ we will denote a parabola from the confocal family $\mathcal{F}$ of parabolas with the focus at $F=(0,0)$ and $x$-axis as the axis of symmetry, as given by the following:
\begin{eqnarray}
\mathcal{D}(E,R)&:&  (x-x_E)^2+(y-y_E)^2=R^2, \quad E\in \mathbb{R}^2, \label{eq.1.0.1}\\
\mathcal{P}(p)&:&  y^2=2px+p^2, \quad p \in \mathbb{R}^*=\mathbb{R}\setminus \{0\}. \label{eq.1.0.2}\\
\mathcal{F}&:&  \{\mathcal{P}(p) \mid ~p \in \mathbb{R}^*\}. \label{eq.1.0.3}
\end{eqnarray}
where $|p|$ is the distance between the focus and the directrix, $\ell: x=-p$, of the parabola.  

With the use of the following transformations:
\begin{equation}\label{eq.1.0.4}
\left(\frac{x}{R},\frac{y}{R}\right) \to (x,y), \qquad \frac{p}{R}\to p,
\end{equation}
it is sufficient to consider $R=1$. We will denote by $\mathcal{D}(E)$ the unit circle centered at $E$.

The organization of this paper is as follows. In Section \ref{sec.2} we review Cayley's theorem which provides the necessary and sufficient conditions for two conics to be an $n$-Poncelet pair for a given $n \geq 3$. For a fixed natural number $n \geq 3$ and a given parabola $\mathcal{P}\in \mathcal{F}$, we use Cayley's theorem in section \ref{sec.2} to describe the locus of the centers $(x,y)$ of $\mathcal{D}$ as an algebraic curve such that for a given parabola $\mathcal{P}(p)$, the pair $(\mathcal{D}(x,y),\mathcal{P}(p))$ is an $n$-Poncelet pair.

In section \ref{sec.3}, Cayley's theorem is used to derive the necessary and sufficient conditions on $x_E, y_E$ and $p$ for $(\mathcal{D}(E),\mathcal{P}(p))$ to be an $n$-Poncelet pair for $n = 3, 4$.

For $n=3$ we also prove that there exist $3$-Poncelet pairs if and only if the circle contains the focus of the parabolas. For $n=4$ we prove that if the center of the circle does not coincide with the focus or does not belong to the line containing the latus rectum, then there exists a unique $4$-Poncelet pair with the circle. And no $4$-Poncelet pair exists if the center of the circle lies on the line containing the latus rectum but does not coincide with the focus.

Moreover, for $n=3$, every parabola from the confocal family forms a $3$-Poncelet pair with the circle if and only if the circle contains the focus.

Similarly, every parabola from the confocal family forms a $4$-Poncelet pair with the circle if and only if the center of the circle coincides with the focus. We refer to such families as $3$-and $4$-isoperiodic families, respectively. 

This motivates us to define the following. 
\begin{definition}\label{def.1.2}
A family $\mathcal{F}$ of nondegenerate conics is said to be an \textit{$n$-isoperiodic with $\mathcal{D}$} if there exist a nondegenerate conic $\mathcal{D}$ and a natural number $n \geq 3$ such that $(\mathcal{D},\mathcal{C})$ is an $n$-Poncelet pair for every nondegenerate conic $\mathcal{C}\in \mathcal{F}$.
\end{definition}

\noindent This concept was introduced and applied in \cite{Dragovic-Radnovic2024} to investigate Poncelet pairs and $n$-isoperiodicity for pairs from a confocal pencil of central conics  with a circle.

We also provide  geometric proofs for $n$-isoperiodicity for $n=3,4$ in section \ref{sec.3}.
Geometric properties of associated Poncelet triangles and quadrilaterals are studied in \cite{Dragovic-Murad2025a}.

In section \ref{sec.4}, for $n = 5$ and $n = 6$, we describe the regions in the $\mathbb{R}^2$-plane that have $0$, $1$, or $2$
parabolas forming $n$-Poncelet pairs $(\mathcal{D}(E), \mathcal{P}(p))$. Similar description of the regions in the $\mathbb{R}^2$-plane is provided that gives $0$, $2$ or $4$ parabolas forming $7$-Poncelet
pairs.

In section \ref{sec.4}, we prove: 
\begin{center}
\textit{A confocal family of parabolas is $n$-isoperiodic with a circle if and only if $n \in \{3,4\}$.}    
\end{center}

Using isoperiodic families that we singled out for $n=3$ and $n=4$, we construct explicit algebraic solutions to Painlev\'e VI equations in section \ref{sec.5}.

The solution we obtained for Painlev\'e VI by using the isoperiodic family for the case of $n=3$ is equivalent to the one that was obtained by Hitchin \cite{Hitchin1995}.

However, we find an intriguing contrast with the situation studied in \cite{Dragovic-Radnovic2024}, where it was proved that isoperiodic families exist, in the case of central confocal conics, exactly for $n=4$ and $n=6$. Moreover, the solutions of Painlev\'e VI equations that we construct here for $n=4$ are different from the corresponding ones obtained in \cite{Dragovic-Radnovic2024}.

\section{Cayley's Theorem}\label{sec.2}
We start this section with Cayley's theorem. It gives a necessary and sufficient condition for two conics to be an $n$-Poncelet pair. We will use the same notation for a conic and for the corresponding $3\times 3$ symmetric matrix that defines it.

\begin{theorem}[Cayley 1853, \cite{Cayley1853a,Cayley1853b}\footnote{This theorem was incorrectly stated in \cite{Cayley1853a} but corrected in \cite{Cayley1853b}.}]\label{thm.2.1}
Let $\mathcal{C}_1$ and $\mathcal{C}_2$ be two distinct conics in $\mathbb{CP}^2$, and
\begin{equation*}
    \sqrt{\det(\lambda\mathcal{C}_1+\mathcal{C}_2)}=A_0+A_1\lambda+\cdots +A_n \lambda^n+\cdots.
\end{equation*}
Then $(\mathcal{C}_1,\mathcal{C}_2)$ is an $n$-Poncelet pair if and only if
\begin{eqnarray}
\left|\begin{array}{ccc}
    A_2 & \cdots & A_{m+1}\\
    \vdots & \ddots & \vdots \\
     A_{m+1} & \cdots & A_{2m}
\end{array}\right| &=& 0,~\text{if}~ n = 2m+1 ~(m \geq 1);\label{eq.2.0.1}\\\nonumber\\
\left|\begin{array}{ccc}
    A_3 & \cdots & A_{m+1}\\
    \vdots & \ddots & \vdots \\
     A_{m+1} & \cdots & A_{2m-1}
\end{array}\right| &=& 0,~\text{if}~ n = 2m ~(m \geq 2)\label{eq.2.0.2}.
\end{eqnarray}
\end{theorem}
For proofs and various variations and generalizations of Cayley's theorem see e.g. \cite{Griffiths-Harris1978,Flatto2009,Dragovic-Radnovic2011}.

\begin{remark}
    Cayley's theorem is based on the consideration that every $k$-gon is also an $n$-gon if $k|n$ for $k,n \geq 3$ and $k<n$. From now on, we will not consider a $k$-gon as a special $mk$-gon where the $k$-gon is  traversed $m$ times, for any $m>1$. For example, a triangle is not a hexagon, traversed twice.
\end{remark}

The characteristic polynomial for the pair $(\mathcal{C}_1,\mathcal{C}_2)$ is given by
\begin{equation}\label{eq.2.0.3}
    f(\lambda):=\det(\lambda\mathcal{C}_1+\mathcal{C}_2)=\Delta_1 \lambda^3+\Theta_1\lambda^2+\Theta_2\lambda+\Delta_2,
\end{equation}
where
\begin{eqnarray*}
    \Delta_1 = \det \mathcal{C}_1,\quad \Theta_1 = \mathrm{tr}(\mathrm{adj}(\mathcal{C}_1)~\mathcal{C}_2), \quad \Theta_2 = \mathrm{tr}(\mathcal{C}_1~\mathrm{adj}(\mathcal{C}_2)), \quad \Delta_2 = \det \mathcal{C}_2.
\end{eqnarray*}

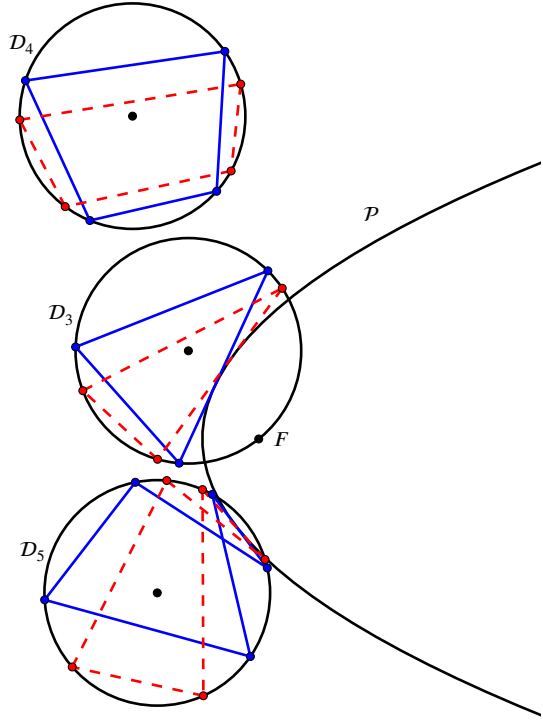
\begin{figure}
    \centering
\definecolor{qqqqff}{rgb}{0.,0.,1.}
\definecolor{ffqqqq}{rgb}{1.,0.,0.}
\begin{tikzpicture}[scale=1.5]
\clip(-2.5,-2.5) rectangle (2.5,4.);
\draw [line width=1.pt] (-0.6234179512971516,0.7818887759782013) circle (1.cm);
\draw [samples=100,rotate around={-90.:(-0.5,0.)},xshift=-0.5cm,yshift=0.cm,line width=1.pt,domain=-6.0:6.0)] plot (\x,{(\x)^2/2/1.0});
\draw [line width=1.pt,color=qqqqff] (-1.6228378603723965,0.8159452782065496)-- (0.08114636169074187,1.4915289163760637);
\draw [line width=1.pt,color=qqqqff] (0.08114636169074187,1.4915289163760637)-- (-0.7051444039126917,-0.21476602230597888);
\draw [line width=1.pt,color=qqqqff] (-0.7051444039126917,-0.21476602230597888)-- (-1.6228378603723965,0.8159452782065496);
\draw [line width=1.pt] (-1.1183562133739073,2.863276083436026) circle (1.cm);
\draw [line width=1.pt,color=qqqqff] (-2.06730763936574,3.1786983260063106)-- (-0.2999701655273033,3.4379449243261653);
\draw [line width=1.pt,color=qqqqff] (-0.2999701655273033,3.4379449243261653)-- (-0.37280356430466965,2.1968293507815386);
\draw [line width=1.pt,color=qqqqff] (-0.37280356430466965,2.1968293507815386)-- (-1.4966310571216137,1.9375827526900125);
\draw [line width=1.pt,color=qqqqff] (-1.4966310571216137,1.9375827526900125)-- (-2.06730763936574,3.1786983260063106);
\draw [line width=1.pt] (-0.8991308467779795,-1.3651422554041102) circle (1.cm);
\draw [line width=1.pt,color=qqqqff] (-1.8972351479880025,-1.4266873925479302)-- (-1.0934641975778934,-0.3842067076911745);
\draw [line width=1.pt,color=qqqqff] (-1.0934641975778934,-0.3842067076911745)-- (0.07609785282640974,-1.1439432873506137);
\draw [line width=1.pt,color=qqqqff] (0.07609785282640974,-1.1439432873506137)-- (-0.41198711880377226,-0.4918204428427891);
\draw [line width=1.pt,color=qqqqff] (-0.41198711880377226,-0.4918204428427891)-- (-0.07379805259746797,-1.9297888518254374);
\draw [line width=1.pt,color=qqqqff] (-0.07379805259746797,-1.9297888518254374)-- (-1.8972351479880025,-1.4266873925479302);
\draw [line width=1.pt,dash pattern=on 4pt off 4pt,color=ffqqqq] (-1.5587673494575651,0.4281635954199716)-- (0.20850659147678305,1.3367775532393562);
\draw [line width=1.pt,dash pattern=on 4pt off 4pt,color=ffqqqq] (0.20850659147678305,1.3367775532393562)-- (-0.89657514461357,-0.1800806350424904);
\draw [line width=1.pt,dash pattern=on 4pt off 4pt,color=ffqqqq] (-0.89657514461357,-0.1800806350424904)-- (-1.5587673494575651,0.4281635954199716);
\draw [line width=1.pt,dash pattern=on 4pt off 4pt,color=ffqqqq] (-2.1178228521428233,2.830619721299904)-- (-0.1595297087091903,3.147268572689601);
\draw [line width=1.pt,dash pattern=on 4pt off 4pt,color=ffqqqq] (-0.1595297087091903,3.147268572689601)-- (-0.2442886617312895,2.377471683326806);
\draw [line width=1.pt,dash pattern=on 4pt off 4pt,color=ffqqqq] (-0.2442886617312895,2.377471683326806)-- (-1.7150712036911653,2.0608228322552025);
\draw [line width=1.pt,dash pattern=on 4pt off 4pt,color=ffqqqq] (-1.7150712036911653,2.0608228322552025)-- (-2.1178228521428233,2.830619721299904);
\draw [line width=1.pt,dash pattern=on 4pt off 4pt,color=ffqqqq] (-1.6527856683892757,-2.0224126790974672)-- (-0.8150648374120749,-0.36868206750431676);
\draw [line width=1.pt,dash pattern=on 4pt off 4pt,color=ffqqqq] (-0.8150648374120749,-0.36868206750431676)-- (0.055990574282660854,-1.0689276861623145);
\draw [line width=1.pt,dash pattern=on 4pt off 4pt,color=ffqqqq] (0.055990574282660854,-1.0689276861623145)-- (-0.4981164993903512,-0.449070482457648);
\draw [line width=1.pt,dash pattern=on 4pt off 4pt,color=ffqqqq] (-0.4981164993903512,-0.449070482457648)-- (-0.49041023366909225,-2.277801808808278);
\draw [line width=1.pt,dash pattern=on 4pt off 4pt,color=ffqqqq] (-0.49041023366909225,-2.277801808808278)-- (-1.6527856683892757,-2.0224126790974672);
\begin{scriptsize}
\draw [fill=black] (0,0) circle (1.0pt);
\draw [fill=black] (-0.6234179512971784,0.7818887759780344) circle (1.0pt);
\draw [fill=qqqqff] (-1.6228378603723965,0.8159452782065496) circle (1.0pt);
\draw [fill=qqqqff] (0.08114636169074187,1.4915289163760637) circle (1.0pt);
\draw [fill=qqqqff] (-0.7051444039126917,-0.21476602230597888) circle (1.0pt);
\draw [fill=black] (-1.118356213373633,2.8632760834362996) circle (1.0pt);
\draw [fill=qqqqff] (-2.06730763936574,3.1786983260063106) circle (1.0pt);
\draw [fill=qqqqff] (-0.2999701655273033,3.4379449243261653) circle (1.0pt);
\draw [fill=qqqqff] (-0.37280356430466965,2.1968293507815386) circle (1.0pt);
\draw [fill=qqqqff] (-1.4966310571216137,1.9375827526900125) circle (1.0pt);
\draw [fill=black] (-0.8991308467772496,-1.365142255403641) circle (1.0pt);
\draw [fill=qqqqff] (-1.8972351479880025,-1.4266873925479302) circle (1.0pt);
\draw [fill=qqqqff] (-1.0934641975778931,-0.3842067076911743) circle (1.0pt);
\draw [fill=qqqqff] (-0.07379805259746797,-1.9297888518254374) circle (1.0pt);
\draw [fill=qqqqff] (0.07609785282640974,-1.1439432873506137) circle (1.0pt);
\draw [fill=qqqqff] (-0.41198711880377226,-0.4918204428427891) circle (1.0pt);
\draw [fill=ffqqqq] (-1.5587673494575651,0.4281635954199716) circle (1.0pt);
\draw [fill=ffqqqq] (0.20850659147678305,1.3367775532393562) circle (1.0pt);
\draw [fill=ffqqqq] (-0.89657514461357,-0.1800806350424904) circle (1.0pt);
\draw [fill=ffqqqq] (-2.1178228521428233,2.830619721299904) circle (1.0pt);
\draw [fill=ffqqqq] (-0.1595297087091903,3.147268572689601) circle (1.0pt);
\draw [fill=ffqqqq] (-0.2442886617312895,2.377471683326806) circle (1.0pt);
\draw [fill=ffqqqq] (-1.7150712036911653,2.0608228322552025) circle (1.0pt);
\draw [fill=ffqqqq] (-1.6527856683892757,-2.0224126790974672) circle (1.0pt);
\draw [fill=ffqqqq] (-0.8150648374120737,-0.36868206750431687) circle (1.0pt);
\draw [fill=ffqqqq] (-0.49041023366909225,-2.277801808808278) circle (1.0pt);
\draw [fill=ffqqqq] (0.055990574282660854,-1.0689276861623145) circle (1.0pt);
\draw [fill=ffqqqq] (-0.4981164993903512,-0.449070482457648) circle (1.0pt);
\draw[color=black] (-1.75,1.1) node {$\mathcal{D}_3$};
\draw[color=black] (-2.1,3.5) node {$\mathcal{D}_4$};
\draw[color=black] (-2,-1) node {$\mathcal{D}_5$};
\draw[color=black] (0.2,0) node {$F$};
\draw[color=black] (1,2) node {$\mathcal{P}$};
\end{scriptsize}
\end{tikzpicture}
    \caption{An illustration of Poncelet's Theorem. $(\mathcal{D}_n,\mathcal{P})$ is a $n$-Poncelet pair. Two $n$-Poncelet polygons are inscribed in $\mathcal{D}_n$ and circumscribed about $\mathcal{P}$.}
    \label{fig.1}
\end{figure}

\subsection{Recursive formula for $A_k$}

\begin{lemma}\label{lemm.2.1}
Let
\begin{equation}\label{eq.2.1.1}
y:=\sqrt{f(\lambda)}=\sum_{k=0}^{+\infty} A_k \lambda^k.
\end{equation}
Then the recursive formula for $A_k$, $k \in \mathbb{N}$ is given by
\begin{eqnarray}\label{eq.2.1.2}
A_k=\left\{\begin{array}{lll}
   \sqrt{\Delta_2}  & if & k=0 \\
   \displaystyle\frac{1}{k!A_0}\tilde{A}_{k}  & if & k\in \mathbb{N}
\end{array}\right.
\end{eqnarray}
where
\begin{eqnarray}
     \tilde{A}_1&=&\frac{1}{2}\Theta_2\label{eq.2.1.3},\\
    \tilde{A}_{k+1}=&&\frac{1}{2}f^{(k+1)}(0)-\frac{1}{A_0^2}\sum_{l=1}^{k} \binom{k}{l}  \tilde{A}_l \tilde{A}_{k-l+1},\quad k \in \mathbb{N}.\label{eq.2.1.4}
\end{eqnarray}
\end{lemma}

\begin{proof} We start from two obvious relations:
\begin{eqnarray}
A_k &=& \frac{y^{(k)}(0)}{k!} \label{eq.2.1.5}\\
y(0)y'(0)&=&\frac{1}{2}f'(0). \label{eq.2.1.6}
\end{eqnarray}
Then, by applying Leibniz's rule to the following product
\begin{equation*}
\left(y y'\right)^{(k)}=\frac{1}{2} f^{(k+1)},
\end{equation*}
and evaluating at $\lambda = 0$, we obtain
\begin{eqnarray*}
y(0)y^{(k+1)}(0)&=&\frac{1}{2}f^{(k+1)}(0)-\sum_{l=1}^{k} \binom{k}{l} y^{(l)}(0) y^{(k-l+1)}(0).
\end{eqnarray*}
Now we use the eq. \eqref{eq.2.1.5} in the previous equation to conclude the proof.
\end{proof}

In particular, the previous lemma gives
\begin{eqnarray*}
\tilde{A}_2&:=&\frac{1}{4\Delta_2}(4 \Delta_2\Theta_1 -\Theta_2^2),\\
\tilde{A}_3&:=&\frac{3}{8\Delta_2^2}(8\Delta_1\Delta_2^2-4\Delta_2\Theta_1\Theta_2+\Theta_2^3).
\end{eqnarray*}
\begin{corollary}\label{cor.2.1}
For $k \geq 3$, the eq. \eqref{eq.2.1.4} reduces to the following:
 \begin{eqnarray}
\tilde{A}_{k+1}&=&-\frac{1}{A_0^2}\sum_{l=1}^{k} \binom{k}{l} \tilde{A}_l \tilde{A}_{k-l+1}.\label{eq.2.1.7}
\end{eqnarray}
\end{corollary}

Using the same symbols $\mathcal{P}(p)$ and $\mathcal{D}(E)$ to denote the matrices associated with the conics (parabolas and circles of radius 1) given by equations \eqref{eq.1.0.1}-\eqref{eq.1.0.2}, we write:
\begin{eqnarray}
\mathcal{P}(p) &=&\begin{pmatrix}
0 & 0 & -p\\
0 & 1 & 0\\
-p & 0 & -p^2
\end{pmatrix},\label{eq.2.1.8}\\
\mathcal{D}(E) &=& \begin{pmatrix}
1 & 0 & -x_E\\
0 & 1 & -y_E\\
-x_E & -y_E & x_E^2+y_E^2-1
\end{pmatrix}.\label{eq.2.1.9}
\end{eqnarray}
It is easy to see that matrices $\mathcal{P}(p)$ and $\mathcal{D}(E)$ are nonsingular for $p \in \mathbb{R}^*$ and for all $E \in \mathbb{R}^2$.

Forming the pencil, we thus obtain
\begin{equation}\label{eq.2.0.10}
 \det(\lambda\mathcal{D}(x_E,y_E)+\mathcal{P}(p))=\Delta_1(x_E,y_E,p) \lambda^3+\Theta_1(x_E,y_E,p)\lambda^2+\Theta_2(x_E,y_E,p)\lambda+\Delta_2(x_E,y_E,p)   
\end{equation}
where
\begin{eqnarray}
    \Delta_1(x_E,y_E,p) &=& -1, \label{eq.2.1.11}\\
    \Theta_1(x_E,y_E,p) &=& -p^2 - 2p x_E + y_E^2 - 1, \label{eq.2.1.12}\\
    \Theta_2(x_E,y_E,p) &=& -2 p^2 - 2p x_E, \label{eq.2.1.13}\\
    \Delta_2(x_E,y_E,p) &=& -p^2. \label{eq.2.1.14}
\end{eqnarray}

We note the following symmetries:
\begin{eqnarray}\label{eq.2.1.15}
    g(-x_E,y_E,-p) = g(x_E,y_E,p) = g(x_E,-y_E,p)
\end{eqnarray}
where $g=\Delta_k,\Theta_k,k=1,2$.

\subsection{Locus of the centers of circles for a given parabola from an $n$-Poncelet pair}\label{sec.2.2}

In eq. \eqref{eq.2.1.2}, we defined
\begin{equation*}
  \tilde{A}_n(p,x,y):=n!A_0 A_n(x,y,p).
\end{equation*}

By the mathematical induction on $n$, it easily follows from the eq. \eqref{eq.2.1.12} and the eqs. \eqref{eq.2.1.3}-\eqref{eq.2.1.4} of Lemma \ref{lemm.2.1} that $\tilde{A}_n(p,x,y) \in \mathbb{R}[x,y]$, and

\begin{lemma}\label{lemm.6.1} The following identities hold:
    \begin{eqnarray*}
\left|\begin{array}{ccc}
    A_2(x,y,p) & \cdots & A_{m+1}(x,y,p)\\
    \vdots & \ddots & \vdots \\
     A_{m+1}(x,y,p) & \cdots & A_{2m}(x,y,p)
\end{array}\right| &=& \frac{1}{A_0^m}\left|\begin{array}{ccc}
    \frac{1}{2!}\tilde{A}_2(p,x,y) & \cdots & \frac{1}{(m+1)!}\tilde{A}_{m+1}(p,x,y)\\
    \vdots & \ddots & \vdots \\
     \frac{1}{(m+1)!}\tilde{A}_{m+1}(p,x,y) & \cdots & \frac{1}{(2m)!}\tilde{A}_{2m}(p,x,y)
\end{array}\right|\\\nonumber\\
\left|\begin{array}{ccc}
    A_3(x,y,p) & \cdots & A_{m+1}(x,y,p)\\
    \vdots & \ddots & \vdots \\
     A_{m+1}(x,y,p) & \cdots & A_{2m-1}(x,y,p)
\end{array}\right| &=& \frac{1}{A_0^m}\left|\begin{array}{ccc}
    \frac{1}{3!}\tilde{A}_3(p,x,y) & \cdots & \frac{1}{(m+1)!}\tilde{A}_{m+1}(p,x,y)\\
    \vdots & \ddots & \vdots \\
     \frac{1}{(m+1)!}\tilde{A}_{m+1}(p,x,y) & \cdots & \frac{1}{(2m-1)!}\tilde{A}_{2m-1}(p,x,y)
\end{array}\right|.
\end{eqnarray*}
\end{lemma}
Define
\begin{eqnarray}\label{eq.2.2.1}
\tilde{\mathcal{Q}}^n(p,x,y):=\left\{\begin{array}{lll}
    \left|\begin{array}{ccc}
    \frac{1}{2!}\tilde{A}_2(p,x,y) & \cdots & \frac{1}{(m+1)!}\tilde{A}_{m+1}(p,x,y)\\
    \vdots & \ddots & \vdots \\
     \frac{1}{(m+1)!}\tilde{A}_{m+1}(p,x,y) & \cdots & \frac{1}{(2m)!}\tilde{A}_{2m}(p,x,y)
\end{array}\right| &\text{if} & n = 2m+1 ~(m \geq 1); \\\\
    \left|\begin{array}{ccc}
    \frac{1}{3!}\tilde{A}_3(p,x,y) & \cdots & \frac{1}{(m+1)!}\tilde{A}_{m+1}(p,x,y)\\
    \vdots & \ddots & \vdots \\
     \frac{1}{(m+1)!}\tilde{A}_{m+1}(p,x,y) & \cdots & \frac{1}{(2m-1)!}\tilde{A}_{2m-1}(p,x,y)
\end{array}\right| & \text{if} & n = 2m ~(m \geq 2).
\end{array}\right.
\end{eqnarray}

\begin{proposition} \label{prop.6.1}
For a given parabola $\mathcal{P}(p)\in \mathcal{F}$, the pair $(\mathcal{D}(x,y),\mathcal{P}(p))$ is an $n$-Poncelet pair if and only if the center $(x,y)$ of $\mathcal{D}$ satisfies the algebraic curve given by
\begin{equation*}
\mathcal{Q}^n(p,x,y):=\frac{\tilde{\mathcal{Q}}^n(p,x,y)}{\prod_{k|n} \tilde{\mathcal{Q}}^k(p,x,y)}=0,\quad 3 \leq k < n.
\end{equation*}
\end{proposition}

\begin{proof} 
Follows from the Theorem \ref{thm.2.1}.
\end{proof}

\begin{example}\label{ex.2.1}
For the parabola $\mathcal{P}(1/2)$,  $(\mathcal{D}(x,y),\mathcal{P}(1/2))$ is a $3$-Poncelet pair or a $4$-Poncelet pair if and only if the center $(x,y)$ of $\mathcal{D}(x,y)$ belongs to the following algebraic curves: 
\begin{eqnarray*}
  Q^3\left(\frac{1}{2}, x,y\right)&:=&  x^2 + y^2 -1=0,\\
  Q^4\left(\frac{1}{2}, x,y\right)&:=&  x^2 + y^2 + 2x(x^2+y^2-1)=0,
\end{eqnarray*}
respectively. See Figure \ref{fig.2}.
\end{example}

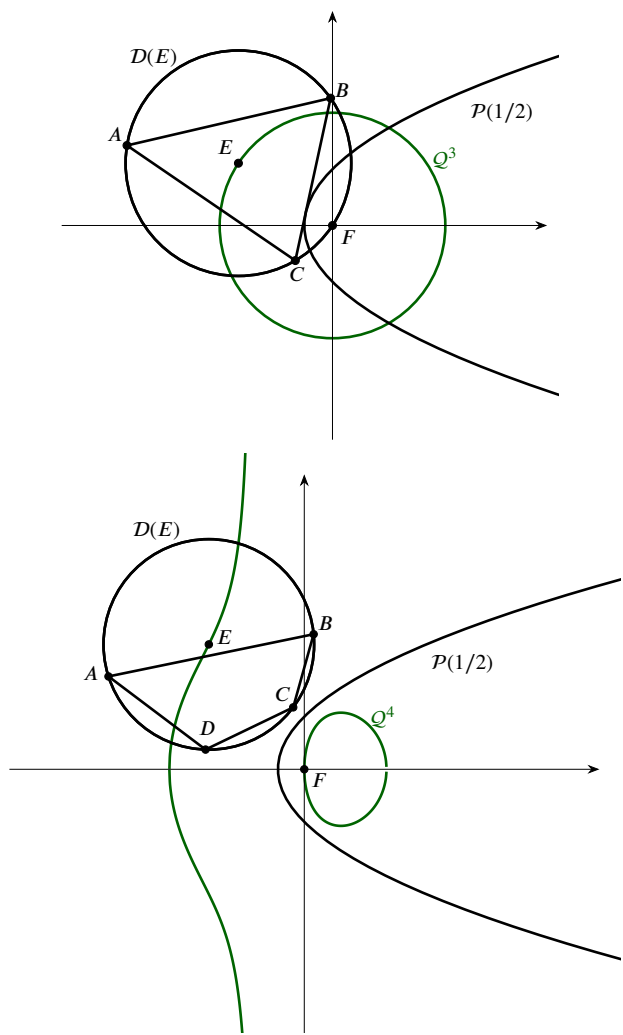
\begin{figure}
    \centering
\definecolor{qqwuqq}{rgb}{0.,0.39215686274509803,0.}
\begin{tikzpicture}[scale=1.5]
\clip(-2.5,-2) rectangle (2,2);
\draw [line width=1.pt,color=qqwuqq] (0.,0.) circle (1.cm);
\draw [samples=100,rotate around={-90.:(-0.25,0.)},xshift=-0.25cm,yshift=0.cm,line width=1.pt,domain=-4.0:4.0)] plot (\x,{(\x)^2/2/0.5});
\draw [line width=1.pt] (-0.833638963731342,0.5523097664795563) circle (1.cm);
\draw [line width=1.pt] (-0.833638963731342,0.5523097664795563) circle (1.cm);
\draw [line width=1.pt] (-1.8211528070302694,0.7098420149018443)-- (-0.0169319484555866,1.1293623214512563);
\draw [line width=1.pt] (-0.0169319484555866,1.1293623214512563)-- (-0.3291931719767479,-0.3111335977204823);
\draw [line width=1.pt] (-0.3291931719767479,-0.3111335977204823)-- (-1.8211528070302694,0.7098420149018443);
\begin{scriptsize}
\draw[color=qqwuqq] (1,0.6) node {$\mathcal{Q}^3$};
\draw[color=black] (1.5,1) node {$\mathcal{P}(1/2)$};
\draw [fill=black] (-0.8336389637315158,0.5523097664795042) circle (1.0pt);
\draw[color=black] (-0.95,0.7) node {$E$};
\draw[color=black] (-1.586172760495931,1.4895050271909835) node {$\mathcal{D}(E)$};
\draw [fill=black] (-1.8211528070302694,0.7098420149018443) circle (1.0pt);
\draw[color=black] (-1.93,0.81) node {$A$};
\draw [fill=black] (-0.0169319484555866,1.1293623214512563) circle (1.0pt);
\draw[color=black] (0.08,1.2) node {$B$};
\draw [fill=black] (-0.3291931719767479,-0.3111335977204823) circle (1.0pt);
\draw[color=black] (-0.31,-0.43) node {$C$};
\draw [fill=black] (0.,0.) circle (1.0pt);
\draw[color=black] (0.14068848205184692,-0.1) node {$F$};
\draw[line width=0.1mm,-{Stealth},color=black] (0,-1.9) -- (0,1.9);
\draw[line width=0.1mm,-{Stealth},color=black] (-2.4,0) -- (1.9,0);
\end{scriptsize}
\end{tikzpicture}

\begin{tikzpicture}[scale=1.4]
\clip(-3,-2.5) rectangle (3,3);
\draw[line width=1.pt,color=qqwuqq,smooth,samples=200,domain=-1.2807752736935358:-0.5000014544792233] plot(\x,{(-sqrt(-4*(\x)^(4)-4*(\x)^(3)+3*(\x)^(2)+2*(\x)))/(2*(\x)+1)});
\draw[line width=1.pt,color=qqwuqq,smooth,samples=200,domain=3.983828253437852E-6:0.7807745858073388] plot(\x,{(-sqrt(-4*(\x)^(4)-4*(\x)^(3)+3*(\x)^(2)+2*(\x)))/(2*(\x)+1)});
\draw[line width=1.pt,color=qqwuqq,smooth,samples=200,domain=-1.2807752736935358:-0.5000014544792233] plot(\x,{sqrt(-4*(\x)^(4)-4*(\x)^(3)+3*(\x)^(2)+2*(\x))/(2*(\x)+1)});
\draw[line width=1.pt,color=qqwuqq,smooth,samples=200,domain=3.983828253437852E-6:0.7807745858073388] plot(\x,{sqrt(-4*(\x)^(4)-4*(\x)^(3)+3*(\x)^(2)+2*(\x))/(2*(\x)+1)});
\draw [samples=100,rotate around={-90.:(-0.25,0.)},xshift=-0.25cm,yshift=0.cm,line width=1.pt,domain=-4.0:4.0)] plot (\x,{(\x)^2/2/0.5});
\draw [line width=1.pt] (-0.9064666527520012,1.1867735335654033) circle (1.cm);
\draw [line width=1.pt] (-0.9064666527520012,1.1867735335654033) circle (1.cm);
\draw [line width=1.pt] (-1.858998275048948,0.8823338703735866)-- (0.08884065753570088,1.283537979969339);
\draw [line width=1.pt] (0.08884065753570088,1.283537979969339)-- (-0.10520342634853108,0.5884615396913863);
\draw [line width=1.pt] (-0.10520342634853108,0.5884615396913863)-- (-0.9375722616422588,0.18725743009554607);
\draw [line width=1.pt] (-0.9375722616422588,0.18725743009554607)-- (-1.858998275048948,0.8823338703735866);
\begin{scriptsize}
\draw[color=qqwuqq] (0.75,0.5) node {$\mathcal{Q}^4$};
\draw[color=black] (1.5,1) node {$\mathcal{P}(1/2)$};
\draw [fill=black] (0,0) circle (1.0pt);
\draw [fill=black] (-0.9064666527516024,1.1867735335660863) circle (1.0pt);
\draw[color=black] (-0.7582356491321075,1.2597043711793439) node {$E$};
\draw[color=black] (-1.4023811147639076,2.264961688756242) node {$\mathcal{D}(E)$};
\draw [fill=black] (-1.858998275048948,0.8823338703735866) circle (1.0pt);
\draw[color=black] (-2.0270070208311073,0.9181120787988443) node {$A$};
\draw [fill=black] (0.08884065753570088,1.283537979969339) circle (1.0pt);
\draw[color=black] (0.19822276953329257,1.3865815083492437) node {$B$};
\draw [fill=black] (-0.10520342634853108,0.5884615396913863) circle (1.0pt);
\draw[color=black] (-0.20192820154100746,0.7131567033705447) node {$C$};
\draw [fill=black] (-0.9375722616422588,0.18725743009554607) circle (1.0pt);
\draw[color=black] (-0.9241519054312075,0.3910839705546452) node {$D$};
\draw[color=black] (0.14068848205184692,-0.1) node {$F$};
\draw[line width=0.1mm,-{Stealth},color=black] (0,-2.8) -- (0,2.8);
\draw[line width=0.1mm,-{Stealth},color=black] (-2.8,0) -- (2.8,0);
\end{scriptsize}
\end{tikzpicture}
    \caption{Example \ref{ex.2.1}: the algebraic curves $Q^3(1/2,x,y)=0$ and $Q^4(1/2,x,y)=0$.}
    \label{fig.2}
\end{figure}

\section{Triangles and Cyclic Quadrilaterals circumscribing Parabolas from a Confocal Family}\label{sec.3}

\subsection{Triangles }\label{sec.3.1}
For $n=3$, we obtain the following:
\begin{proposition}\label{prop.3.1}
A circle $\mathcal{D}(E)$ and a parabola $\mathcal{P}(p)\in \mathcal{F}$ form a $3$-Poncelet pair $(\mathcal{D}(E),\mathcal{P}(p))$ if and only if 
\begin{equation}\label{eq.3.1.1}
\mathcal{Q}^3_0(x_E,y_E,p):= x_E^2 + y_E^2 - 1=0.
\end{equation}
\end{proposition}

\begin{proof} By using eq. \eqref{eq.2.1.2} together with eqs. \eqref{eq.2.1.11}-\eqref{eq.2.1.14}, we obtain
\begin{eqnarray}\label{eq.3.1.2}
A_2(x_E,y_E,p)=-\frac{\sqrt{-p^2}}{2p^2}\mathcal{Q}^3_0(x_E,y_E,p).
\end{eqnarray}
Therefore, the Cayley condition, eq. \eqref{eq.2.0.1} for $n=3$, given by $A_2(x_E,y_E,p)=0$ is equivalent to $\mathcal{Q}^3_0(x_E,y_E,p)=0$. 
\end{proof}

Denote
\begin{equation*}
S^1=\{(x,y)\in \mathbb{R}^2~|~ x^2+y^2=1\}.
\end{equation*}
We note that $E\in S^1$ is equivalent to the condition that the circle $\mathcal{D}(E)$ contains the focus $F=(0,0)$ of the confocal family $\mathcal{F}$. This gives
\begin{theorem}\label{thm.3.1}
A circle $\mathcal{D}$ and a parabola $\mathcal{P}$ with the focus at a point $F$ form a $3$-Poncelet pair if and only if $F\in \mathcal{D}$.
\end{theorem}
\begin{proof}
Take $F=(0,0)$, denote by $E=(x_E,y_E)$ the center of $\mathcal{D}$ and $\mathcal{P}=\mathcal{P}(p)$. The rest now follows from the Proposition \ref{prop.3.1}. See Figure \ref{fig.3}.
\end{proof}

\begin{remark}\label{rem.3.1}
The circle $\mathcal{D}(E)$ contains $F$ if and only if $E$ lies on a circle $\tilde{\mathcal{D}}(F)$, congruent to $\mathcal{D}$. That is, $F \in \mathcal{D}(E) \Leftrightarrow E \in \tilde{\mathcal{D}}(F)$. See Figure \ref{fig.3}.
\end{remark}
\begin{remark}\label{rem.3.2}
Poncelet triangles corresponding to  real $3$-Poncelet pairs $(\mathcal{D}(E),\mathcal{P}(p))$ that satisfy the eq. \eqref{eq.3.1.1} generally exist in $\mathbb{CP}^2$. See Figure \ref{fig.3}.
\end{remark}
\begin{figure}
    \centering
\begin{tikzpicture}[scale=1.5]
\clip(-2.5,-2) rectangle (2.5,2);
\draw [line width=1.0pt,dash pattern=on 3pt off 3pt] (0.,0.) circle (1.cm);
\draw [samples=100,rotate around={-90.:(-0.5,0.)},xshift=-0.5cm,yshift=0.cm,line width=1.0pt,domain=-6.0:6.0)] plot (\x,{(\x)^2/2/1.0});
\draw [line width=1.0pt] (-1.,0.) circle (1.cm);
\draw [line width=1.0pt] (1.,0.) circle (1.cm);
\draw [line width=1.0pt] (-1.,0.) circle (1.cm);
\draw [line width=1.0pt] (1.,0.) circle (1.cm);
\draw [line width=1.0pt] (-1.9558058112196797,0.2939987265936541)-- (-0.6522365335520236,0.9375823011363255);
\draw [line width=1.0pt] (-0.6522365335520236,0.9375823011363255)-- (-0.39195765522829673,-0.7939045956313195);
\draw [line width=1.0pt] (-0.39195765522829673,-0.7939045956313195)-- (-1.9558058112196797,0.2939987265936541);
\begin{scriptsize}
\draw [fill=black] (-1.,0.) circle (1.0pt);
\draw[color=black] (-0.85,0.16703142999011367) node {$E$};
\draw[color=black] (1.7006952858498614,1.934149623493163) node {$\mathcal{P}$};
\draw [fill=black] (1.,0.) circle (1.0pt);
\draw[color=black] (1.1,0.16703142999011367) node {$E'$};
\draw[color=black] (-1.5,1) node {$\mathcal{D}$};
\draw[color=black] (1.1997081691350782,1.141679093417053) node {$\mathcal{D}'$};
\draw [fill=black] (0.,0.) circle (1.0pt);
\draw[color=black] (0.1,0.16703142999011367) node {$F$};
\draw[color=black] (-0.1,1.14) node {$\tilde{\mathcal{D}}$};
\draw [fill=black] (-1.9558058112196797,0.2939987265936541) circle (1.0pt);
\draw [fill=black] (-0.6522365335520236,0.9375823011363255) circle (1.0pt);
\draw [fill=black] (-0.39195765522829673,-0.7939045956313195) circle (1.0pt);
\end{scriptsize}
\end{tikzpicture}
    \caption{Both $\mathcal{D}$ and $\mathcal{D}'$ contain the focus $F$ of $\mathcal{P}$. Two 3-Poncelet pairs $(\mathcal{D},\mathcal{P})$ and $(\mathcal{D}',\mathcal{P})$ with real and imaginary 3-Poncelet triangles, respectively.}
    \label{fig.3}
\end{figure}
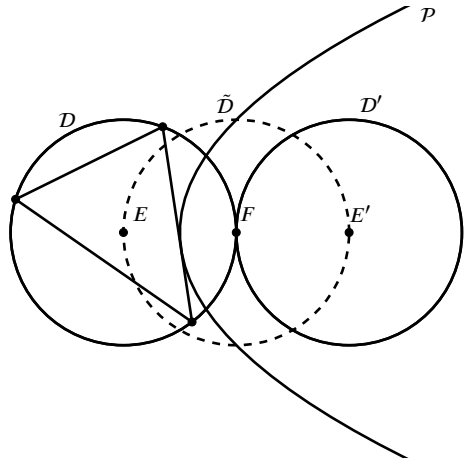
Also, we observe that the equation \eqref{eq.3.1.1} is satisfied for all $p \in \mathbb{R}^*$ if and only if $E\in S^1$. \noindent
This proves the following:
\begin{theorem}\label{thm.3.2}
The confocal family of parabolas $\mathcal{F}$ is $3$-isoperiodic with $\mathcal{D}(E)$ if and only if \newline $F\in \mathcal{D}(E)$.
\end{theorem}
\begin{proof}
 With $F=(0,0)$ and $E=(x_E,y_E)$, $F\in \mathcal{D}(E)$ if and only if $x_E^2+y_E^2=1$ which is the necessary and sufficient condition for $(\mathcal{D}(E),\mathcal{P})$ to be a 3-Poncelet pair by the Proposition \ref{prop.3.1}. Since the equation holds true for every $p\in \mathbb{R}^*$, the confocal family $\mathcal{F}$ is 3-isoperiodic with $\mathcal{D}(E)$. See Figure \ref{fig.4}.
\end{proof}

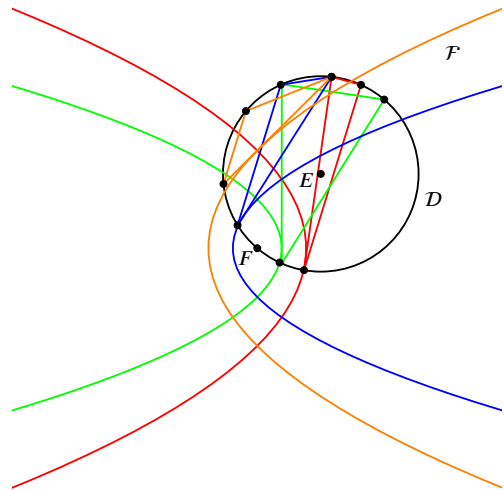
\begin{figure}
\centering
\definecolor{qqqqff}{rgb}{0.,0.,1.}
\definecolor{xdxdff}{rgb}{0.49019607843137253,0.49019607843137253,1.}
\definecolor{yqyqyq}{rgb}{0.5019607843137255,0.5019607843137255,0.5019607843137255}
\definecolor{ffqqqq}{rgb}{1.,0.,0.}
\definecolor{xdxdff}{rgb}{0.49019607843137253,0.49019607843137253,1.}
\definecolor{yqyqyq}{rgb}{0.5019607843137255,0.5019607843137255,0.5019607843137255}
\begin{tikzpicture}[scale=1.3]
\clip(-2.5,-2.5) rectangle (2.5,2.5);
\draw [line width=0.7pt] (0.6494995268592774,0.7603619957683008) circle (1.cm);
\draw [samples=100,rotate around={-270.:(0.5,0.)},xshift=0.5cm,yshift=0.cm,line width=0.7pt,color=ffqqqq,domain=-6.0:6.0)] plot (\x,{(\x)^2/2/1.0});
\draw [line width=0.7pt,color=ffqqqq] (0.47810607729852256,-0.22484066635135924)-- (1.0605984402760933,1.6719527391326205);
\draw [line width=0.7pt,color=ffqqqq] (0.7602945361439927,1.7542052761287406)-- (0.47810607729852256,-0.22484066635135924);
\draw [line width=0.7pt,color=ffqqqq] (0.7602945361439927,1.7542052761287406)-- (1.0605984402760933,1.6719527391326205);
\draw [samples=100,rotate around={-270.:(0.25,0.)},xshift=0.25cm,yshift=0.cm,line width=0.7pt,color=green,domain=-3.0:3.0)] plot (\x,{(\x)^2/2/0.5});
\draw [line width=0.7pt,color=green] (0.24992573762915582,-0.1563390388990996)-- (0.2507974620789881,1.6774425067050092);
\draw [line width=0.7pt,color=green] (1.2982758540102473,1.521341152728922)-- (0.24992573762915582,-0.1563390388990996);
\draw [line width=0.7pt,color=green] (1.2982758540102473,1.521341152728922)-- (0.2507974620789881,1.6774425067050092);
\draw [samples=100,rotate around={-90.:(-0.25,0.)},xshift=-0.25cm,yshift=0.cm,line width=0.7pt,color=blue,domain=-4.0:4.0)] plot (\x,{(\x)^2/2/0.5});
\draw [line width=0.7pt,color=blue] (0.24010408933620908,1.6727190397676857)-- (-0.20139957176193027,0.23503283803420325);
\draw [line width=0.7pt,color=blue] (0.7602945361439927,1.7542052761287406)-- (0.24010408933620908,1.6727190397676857);
\draw [line width=0.7pt,color=blue] (0.7602945361439927,1.7542052761287406)-- (-0.20139957176193027,0.23503283803420325);
\draw [samples=100,rotate around={-90.:(-0.5,0.)},xshift=-0.5cm,yshift=0.cm,line width=0.7pt,color=orange,domain=-6.0:6.0)] plot (\x,{(\x)^2/2/1.0});
\draw [line width=0.7pt,color=orange] (-0.11608848744035058,1.403693164266107)-- (-0.34520699498548285,0.6576053152093247);
\draw [line width=0.7pt,color=orange] (0.7602945361439927,1.7542052761287406)-- (-0.11608848744035058,1.403693164266107);
\draw [line width=0.7pt,color=orange] (0.7602945361439927,1.7542052761287406)-- (-0.34520699498548285,0.6576053152093247);
\begin{scriptsize}
\draw [fill=black] (0,0) circle (1.pt);
\draw [fill=black] (0.649499526859468,0.760361995768678) circle (1.pt);
\draw [fill=black] (0.7602945361439927,1.7542052761287406) circle (1.pt);
\draw [fill=black] (0.7602945361439927,1.7542052761287406) circle (1.pt);
\draw [fill=black] (0.47810607729852256,-0.22484066635135924) circle (1.pt);
\draw [fill=black] (1.0605984402760933,1.6719527391326205) circle (1.pt);
\draw [fill=black] (0.22852305032599546,-0.1467095594056133) circle (1.pt);
\draw [fill=black] (0.24010408933620908,1.6727190397676857) circle (1.pt);
\draw [fill=black] (-0.20139957176193027,0.23503283803420325) circle (1.pt);
\draw [fill=black] (-0.11608848744035058,1.403693164266107) circle (1.pt);
\draw [fill=black] (-0.34520699498548285,0.6576053152093247) circle (1.pt);
\draw [fill=black] (1.2982758540102473,1.521341152728922) circle (1.pt);
\draw[color=black] (1.8,0.5) node {$\mathcal{D}$};
\draw[color=black] (2,2) node {$\mathcal{F}$};
\draw[color=black] (0.5,0.7) node {$E$};
\draw[color=black] (-0.12,-0.1) node {$F$};
\end{scriptsize}
\end{tikzpicture}
\caption{An illustration of $3$-isoperiodicity of $\mathcal{F}$ with $\mathcal{D}(E)$.}\label{fig.4}
\end{figure}

\begin{corollary}\label{cor.3.1} A circle and a parabola do not form an $n$-Poncelet pair $(\mathcal{D}(E),\mathcal{P})$ for $n\neq 3$ and for any $\mathcal{P}\in \mathcal{F}$ if $F \in \mathcal{D}(E)$. 
\end{corollary}

\begin{proof}
Let us assume that $F \in \mathcal{D}(E)$ and $(\mathcal{D}(E),\mathcal{P})$ be an $n$-Poncelet pair for $n\neq 3$ and for some $\mathcal{P}\in \mathcal{F}$. Then by Theorem \ref{thm.3.2}, $(\mathcal{D}(E),\mathcal{P})$ is a $3$-Poncelet pair for every $\mathcal{P}\in \mathcal{F}$. But this contradicts our hypothesis that $n\neq 3$.
\end{proof}

\subsubsection{Geometric Approach to $3$-isoperiodicity}
We will use the following well-known properties of a parabola to prove the Theorem \ref{thm.3.3}.

\begin{lemma}[Two Properties of Parabola]\label{lemm.3.1}
Let $X$ be an arbitrary point on a parabola $\mathcal{P}$.
\begin{itemize}
\item[(i)] [Defining property]
There exist a point $F$ (the focus) and a line $\ell$ (the directrix) such that $X$ is equidistant from $F$ and $\ell$, that is, $\mathrm{dist}(X,F)=\mathrm{dist}(X,\ell)$.
\item[(ii)] [Focal property]  The tangent to the parabola $\mathcal{P}$ at $X$ is the bisector of the angle between $FX$ and the perpendicular from $X$ to $\ell$. 
\end{itemize}
\end{lemma}

\begin{theorem}\label{thm.3.3}
Given a parabola $\mathcal{P}$ with the focus $F$. Let $\mathcal{D}$ be the circle which contains $F$ with the center $E$ on the axis of $\mathcal{P}$. Then, $(\mathcal{D}, \mathcal{P})$ is a $3$-Poncelet pair. Moreover, $(\mathcal{D}, \mathcal{P'})$ is a $3$-Poncelet pair if $\mathcal{P}'$is confocal with $\mathcal{P}$. 
\end{theorem}

\begin{proof}
Let $D$ be the vertex and $\ell$ be the directrix of the parabola $\mathcal{P}$, and $|DF|=x$ as shown in the Figure \ref{fig.5}. Let $E$ be the center of the circle $\mathcal{D}(E)$ and $F\in \mathcal{D}(E)$. Denote $|DE|=y$ and let $A$ be the intersection of the line $DF$ with $\mathcal{D}$, outside $\mathcal{P}$. Denote by $T,T'$, the points of contact of the tangents through $A$ to $\mathcal{P}$. Denote by $B$ the intersection of $\mathcal{D}$ with $AT$. Then 
\begin{equation*}
    |EA|=|EB|=x+y.
\end{equation*}
We want to show that $BD$ is tangent to $\mathcal{P}$ at $D$, equivalently, $BD \perp DF$. By the focal property (ii) from Lemma \ref{lemm.3.1}, $\angle ATF=\angle ATH$. Since $EA \cong EB$, we have $\angle EAB = \angle EBA$. Also, $\angle EAB = \angle ATH$ (alternate angles.) Thus, $\angle FTA=\angle EBA$. Hence, $EB \parallel FT$. 

From the similarity of $\triangle AEB$ and $\triangle AFT$, we get $|FT|=2(x+y)$.

Let $G$ be the orthogonal projection from $F$ to $TT'$. Since $\mathrm{dist}(F,\ell)=2x$, we, thus, have $|FG|=2y$. From there, it follows that $\triangle EDB \sim \triangle FGT$. Since $\triangle FGT$ is a right triangle, we get $\angle EDB=\angle FGT=90^\circ$. This proves that $BD$ is tangent to $\mathcal{P}$ at $D$.

Denote by $C$ the intersection of $\mathcal{D}$ with $AT'$, then by symmetry, $CD$ is tangent to $\mathcal{P}$ at $D$. Thus, $B,D,C$ are collinear. We get that $\triangle ABC$ is a Poncelet triangle inscribed in $\mathcal{D}$ and circumscribed about $\mathcal{P}$.

By varying $x$, one gets different parabolas $\mathcal{P}'$, from the confocal family $\mathcal{F}$ that contains $\mathcal{P}$, with the same property.
\end{proof}

\begin{figure}
    \centering
\definecolor{uuuuuu}{rgb}{0.26666666666666666,0.26666666666666666,0.26666666666666666}
\begin{tikzpicture}[scale=1.5]
\clip(-2.5,-2) rectangle (2.5,2);
\draw [line width=1.pt] (-1.,0.) circle (1.cm);
\draw [samples=100,rotate around={-90.:(-0.25,0.)},xshift=-0.25cm,yshift=0.cm,line width=1.pt,domain=-3.0:3.0)] plot (\x,{(\x)^2/2/0.5});
\draw [line width=1.pt,dash pattern=on 4pt off 4pt,domain=-2.8917829569443203:3.950127171467037] plot(\x,{(-2.6457513110645907-1.3228756555322954*\x)/3.5});
\draw [line width=1.pt,dash pattern=on 4pt off 4pt,domain=-2.8917829569443203:3.950127171467037] plot(\x,{(--2.6457513110645903--1.3228756555322951*\x)/3.5});
\draw [line width=1.pt] (-0.5,-2.0495640515345803) -- (-0.5,3.0414635637006433);
\draw [line width=1.pt] (-0.25,0.6614378277661477)-- (-0.25,0.);
\draw [line width=1.pt] (0.,0.)-- (1.5,1.3228756555322951);
\draw [line width=1.pt,domain=-2.8917829569443203:3.950127171467037] plot(\x,{(--1.3228756555322951-0.*\x)/1.});
\draw [line width=1.pt] (-0.25,-0.6614378277661477)-- (-0.25,0.);
\draw [line width=1.pt,domain=-2.8917829569443203:3.950127171467037] plot(\x,{(-0.-0.*\x)/1.});
\draw [line width=1.pt] (-1.,0.)-- (-0.25,0.6614378277661477);
\draw [line width=1.pt] (1.5,1.3228756555322951)-- (1.5,0.);
\draw [line width=1.pt] (1.5,0.)-- (1.5,-1.3228756555322954);
\begin{scriptsize}
\draw[color=black] (-1.47,1) node {$\mathcal{D}$};
\draw[color=black] (0.3877546735048245,0.65) node {$\mathcal{P}$};
\draw [fill=black] (-1.,0.) circle (1.0pt);
\draw[color=black] (-1,0.2) node {$E$};
\draw [fill=black] (-2.,0.) circle (1.0pt);
\draw[color=black] (-2.09,0.13) node {$A$};
\draw [fill=black] (1.5,1.3228756555322951) circle (1.0pt);
\draw[color=black] (1.5157270618009882,1.506074282915417) node {$T$};
\draw [fill=black] (1.5,-1.3228756555322954) circle (1.0pt);
\draw[color=black] (1.5628662063864995,-1.5) node {$T'$};
\draw [fill=black] (-0.25,0.) circle (1.0pt);
\draw[color=black] (-0.365,0.1) node {$D$};
\draw [fill=black] (-0.25,0.6614378277661477) circle (1.0pt);
\draw[color=black] (-0.20485171557017495,0.8) node {$B$};
\draw[color=black] (-0.38667413040000437,1.9) node {$\ell$};
\draw [fill=black] (0.,0.) circle (1.0pt);
\draw[color=black] (0.1,-0.08) node {$F$};
\draw [fill=uuuuuu] (-0.5,1.3228756555322951) circle (1.0pt);
\draw[color=uuuuuu] (-0.35,1.4522009748176898) node {$H$};
\draw [fill=uuuuuu] (-0.25,-0.6614378277661477) circle (1.0pt);
\draw[color=uuuuuu] (-0.26545918718011813,-0.8) node {$C$};
\draw [fill=uuuuuu] (1.5,0.) circle (1.0pt);
\draw[color=uuuuuu] (1.7,0.12557076291115674) node {$G$};
\draw[color=black] (-0.67,-0.14) node {$y$};
\draw[color=black] (-0.12,-0.14) node {$x$};
\end{scriptsize}
\end{tikzpicture}
    \caption{A geometric proof of $3$-isoperiodicity.}
    \label{fig.5}
\end{figure}

\subsection{Quadrilaterals}\label{sec.3.2}
Using Cayley's Theorem for $n=4$, we get:
\begin{proposition} \label{prop.3.2} A circle $\mathcal{D}(E)$ and a parabola $\mathcal{P}(p)$ form a $4$-Poncelet pair $(\mathcal{D}(E),\mathcal{P}(p))$ if and only if $E\in \mathbb{R}^2\setminus S^1$ and
\begin{equation}\label{eq.3.2.1}
\mathcal{Q}^4_{1}(x_E,y_E,p):=(x_E^2 + y_E^2)p + x_E(x_E^2+y_E^2-1)=0.
\end{equation}
\end{proposition}
\begin{proof} 
By Corollary \ref{cor.3.1}, if $\left(\mathcal{D}(E),\mathcal{P}(p)\right)$ is a $4$-Poncelet pair, then $E\in \mathbb{R}^2\setminus S^1$.
Now we calculate $A_3(x_E,y_E,p)$:
\begin{eqnarray*}
A_3(x_E,y_E,p)=-\frac{\sqrt{-p^2}}{2p^3}\mathcal{Q}^4_1(x_E,y_E,p).
\end{eqnarray*}
Therefore, the Cayley condition, eq. \eqref{eq.2.0.2}, for $n=4$, given by $A_3(x_E,y_E,p)=0$ is equivalent to $\mathcal{Q}^4_1(x_E,y_E,p)=0$.
\end{proof}

\begin{theorem}\label{thm.3.4}
Let $F$ be the focus of the  confocal family of parabolas $\mathcal{F}$. Let $\mathcal{D}$ be a circle centered at $E$. Let $\ell'$ be the line containing $F$ parallel to the directrix $\ell$ (Figure \ref{fig.6}). Then
\begin{itemize}
    \item[a)] for $E \in \ell'$, $(\mathcal{D}(E),\mathcal{P})$ is a $4$-Poncelet pair if and only if $E=F$.
    \item[b)] $(\mathcal{D}(E),\mathcal{P})$ is a 4-Poncelet pair for a unique $\mathcal{P}\in \mathcal{F}$ if and only if $E \notin \ell'$ and $F\notin \mathcal{D}$;
\end{itemize}
\end{theorem}
\begin{proof}
With $F=(0,0)$ and $E=(x_E,y_E)$, we have $\ell': x=0$.

If $E \in \ell'$, $x_E=0$, then the eq. \eqref{eq.3.2.1} can only be satisfied for $p \neq 0$ if and only if $y_E=0$. Hence, there exists no $4$-Poncelet pair if $x_E=0$ and $y_E \neq 0$. This proves a).

To prove b), let $E \notin \ell'$ and $F \notin \mathcal{D}$. Then $(x_E, y_E)\in \mathbb{R}^2\setminus S^1$ with $x_E \neq 0$. By the Proposition \ref{prop.3.2}, there exists a unique nonzero $p^*$ such that $\left(\mathcal{D}(E),\mathcal{P}(p^*)\right)$ is a $4$-Poncelet pair where
\begin{eqnarray}\label{eq.3.2.2}
p^*(E) =-\frac{x_E(x_E^2+y_E^2-1)}{x_E^2 + y_E^2}.
\end{eqnarray}
See Figure \ref{fig.6}.
\end{proof}

\begin{figure}
    \centering
\begin{tikzpicture}[scale=1.5]
\clip(-2.5,-2) rectangle (2,4);
\draw [samples=100,rotate around={-90.:(-0.25,0.)},xshift=-0.25cm,yshift=0.cm,line width=1.pt,domain=-4.0:4.0)] plot (\x,{(\x)^2/2/0.5});

\draw [line width=1.pt,dash pattern=on 4pt off 4pt] (0.,0.) circle (1.cm);
\draw [line width=1.pt,dash pattern=on 4pt off 4pt] (0.,-2.169447815275758) -- (0.,4.501737632813708);
\draw [line width=1.pt] (-0.5,-2.169447815275758) -- (-0.5,4.501737632813708);
\draw [line width=1.pt] (-0.5,-2.169447815275758) -- (-0.5,4.501737632813708);
\draw [line width=1.pt] (-0.90647, 1.1868) circle (1.cm);
\draw [line width=1.pt] (0.,2.5) circle (1.cm);
\draw [line width=1.pt] (-0.8893287413672115,-0.45726840015392617) circle (1.cm);
\draw [line width=1.pt,domain=-5.242608256494929:3.722900229297295] plot(\x,{(-0.-0.*\x)/1.});
\begin{scriptsize}
\draw [fill=black] (0.,0.) circle (1.0pt);
\draw[color=black] (0.09,0.14693602086641783) node {$F$};
\draw[color=black] (-0.6716108198410193,3.7) node {$\ell$};
\draw[color=black] (0.22,3.7) node {$\ell'$};
\draw [fill=black] (0.,2.5) circle (1.0pt);
\draw[color=black] (0.1,2.6) node {$E$};
\draw[color=black] (-0.9363404011144117,3.2354478023893187) node {$\mathcal{D}$};
\draw [fill=black] (-0.90647, 1.1868) circle (1.0pt);
\draw[color=black] (-1.05, 1.1868) node {$E$};
\draw[color=black] (-1.9423128099533031,1.6647189535005291) node {$\mathcal{D}$};
\draw [fill=black] (-0.8893287413665716,-0.45726840015394626) circle (1.0pt);
\draw[color=black] (-0.8,-0.29427994792256795) node {$E$};
\draw[color=black] (-2,-0.2263548952484353) node {$\mathcal{D}$};
\draw[color=black] (0.9,0.7) node {$\tilde{\mathcal{D}}$};
\draw[color=black] (1.8,1.5) node {$\mathcal{P}$};
\end{scriptsize}
\end{tikzpicture}
\caption{Theorem \ref{thm.3.4}: no 4-Poncelet pair exists if $E\in \ell' \cup \tilde{\mathcal{D}}\setminus \{F\}$.}\label{fig.6}
\end{figure}
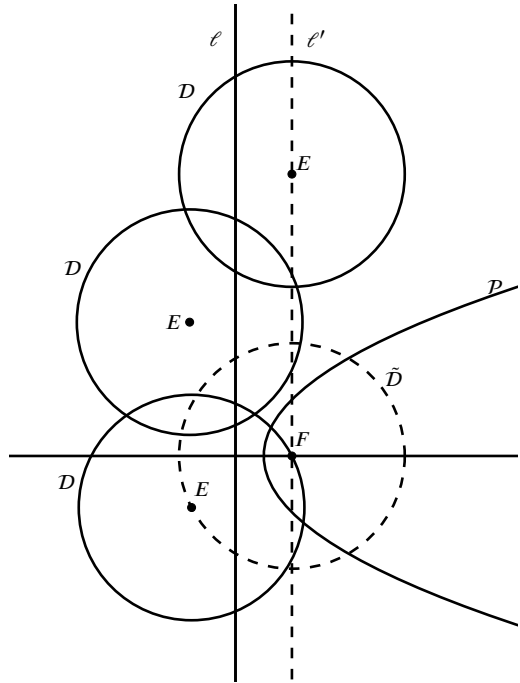

Observe that $\mathcal{Q}^4_{1}(0,0,p)=0$ for every $p\neq 0$. Thus, we also have the following:

\begin{theorem}\label{thm.3.5}
The confocal family  $\mathcal{F}$ of parabolas with the focus $F$ is $4$-isoperiodic with $\mathcal{D}(E)$ if and only if $E=F$.
\end{theorem}
\begin{proof}
The confocal family $\mathcal{F}$ with the focus $F=(0,0)$ is $4$-isoperiodic with $\mathcal{D}(E)$ if and only if the polynomial $\mathcal{Q}^4_{1}(x_E,y_E,p)=0$---the zero polynomial in $\mathbb{R}[p]$. By the eq. \eqref{eq.3.2.1}, this happens if and only if $(x_E,y_E)=(0,0)$. See Figure \ref{fig.7}. 
\end{proof}

\begin{figure}
\centering
\definecolor{yqyqyq}{rgb}{0.5019607843137255,0.5019607843137255,0.5019607843137255}
\begin{tikzpicture}[scale=1.7]
\clip(-1.5,-2.5) rectangle (2.5,2.5);
\draw [samples=100,rotate around={-90.:(-0.5,0.)},xshift=-0.5cm,yshift=0.cm,line width=0.7pt,color=yqyqyq,domain=-6.0:6.0)] plot (\x,{(\x)^2/2/1.0});
\draw [line width=0.7pt] (0.,0.) circle (1.cm);
\draw [line width=0.7pt,color=black] (-0.5806227727406026,-0.8141727063559762)-- (-0.419377227259398,-0.907812062739981);
\draw [line width=0.7pt,color=black] (-0.419377227259398,-0.907812062739981)-- (-0.5806227727406013,-0.814172706355977);
\draw [line width=0.7pt,color=black] (-0.5806227727406013,-0.814172706355977)-- (-0.41937722725939736,0.9078120627399813);
\draw [line width=0.7pt,color=black] (-0.419377227259398,-0.907812062739981)-- (-0.5806227727406019,0.8141727063559766);
\draw [line width=0.7pt,color=black] (-0.5806227727406019,0.8141727063559766)-- (-0.41937722725939736,0.9078120627399813);
\draw [samples=100,rotate around={-90.:(-0.4,0.)},xshift=-0.4cm,yshift=0.cm,line width=0.7pt,color=yqyqyq,domain=-4.800000000000001:4.800000000000001)] plot (\x,{(\x)^2/2/0.8});
\draw [line width=0.7pt,color=black] (-0.5806227727406026,-0.8141727063559762)-- (-0.21937722725939784,-0.9756401140584465);
\draw [line width=0.7pt,color=black] (-0.21937722725939784,-0.9756401140584465)-- (-0.5806227727406023,-0.8141727063559765);
\draw [line width=0.7pt,color=black] (-0.5806227727406023,-0.8141727063559765)-- (-0.21937722725939746,0.9756401140584465);
\draw [line width=0.7pt,color=black] (-0.21937722725939784,-0.9756401140584465)-- (-0.5806227727406021,0.8141727063559765);
\draw [line width=0.7pt,color=black] (-0.5806227727406021,0.8141727063559765)-- (-0.21937722725939746,0.9756401140584465);
\draw [samples=100,rotate around={-90.:(-0.4,0.)},xshift=-0.4cm,yshift=0.cm,line width=0.7pt,color=yqyqyq,domain=-4.800000000000001:4.800000000000001)] plot (\x,{(\x)^2/2/0.8});
\draw [line width=0.7pt,color=black] (-0.5806227727406026,-0.8141727063559762)-- (-0.21937722725939784,-0.9756401140584465);
\draw [line width=0.7pt,color=black] (-0.21937722725939784,-0.9756401140584465)-- (-0.5806227727406023,-0.8141727063559765);
\draw [line width=0.7pt,color=black] (-0.5806227727406023,-0.8141727063559765)-- (-0.21937722725939746,0.9756401140584465);
\draw [line width=0.7pt,color=black] (-0.21937722725939784,-0.9756401140584465)-- (-0.5806227727406021,0.8141727063559765);
\draw [line width=0.7pt,color=black] (-0.5806227727406021,0.8141727063559765)-- (-0.21937722725939746,0.9756401140584465);
\draw [samples=100,rotate around={-90.:(-0.3,0.)},xshift=-0.3cm,yshift=0.cm,line width=0.7pt,color=yqyqyq,domain=-3.5999999999999996:3.5999999999999996)] plot (\x,{(\x)^2/2/0.6});
\draw [line width=0.7pt,color=black] (-0.5806227727406026,-0.8141727063559762)-- (-0.019377227259398083,-0.9998122439056933);
\draw [line width=0.7pt,color=black] (-0.019377227259398083,-0.9998122439056933)-- (-0.5806227727406011,-0.8141727063559772);
\draw [line width=0.7pt,color=black] (-0.5806227727406011,-0.8141727063559772)-- (-0.01937722725939739,0.9998122439056933);
\draw [line width=0.7pt,color=black] (-0.019377227259398083,-0.9998122439056933)-- (-0.580622772740602,0.8141727063559765);
\draw [line width=0.7pt,color=black] (-0.580622772740602,0.8141727063559765)-- (-0.01937722725939739,0.9998122439056933);
\draw [samples=100,rotate around={-90.:(-0.2,0.)},xshift=-0.2cm,yshift=0.cm,line width=0.7pt,color=yqyqyq,domain=-3.2:3.2)] plot (\x,{(\x)^2/2/0.4});
\draw [line width=0.7pt,color=black] (-0.5806227727406026,-0.8141727063559762)-- (0.18062277274060312,-0.9835524459669126);
\draw [line width=0.7pt,color=black] (0.18062277274060312,-0.9835524459669126)-- (-0.5806227727406029,-0.8141727063559759);
\draw [line width=0.7pt,color=black] (-0.5806227727406029,-0.8141727063559759)-- (0.18062277274060257,0.9835524459669127);
\draw [line width=0.7pt,color=black] (0.18062277274060312,-0.9835524459669126)-- (-0.5806227727406031,0.8141727063559758);
\draw [line width=0.7pt,color=black] (-0.5806227727406031,0.8141727063559758)-- (0.18062277274060257,0.9835524459669127);
\draw [samples=100,rotate around={-90.:(-0.1,0.)},xshift=-0.1cm,yshift=0.cm,line width=0.7pt,color=yqyqyq,domain=-2.0:2.0)] plot (\x,{(\x)^2/2/0.2});
\draw [line width=0.7pt,color=black] (-0.5806227727406026,-0.8141727063559762)-- (0.3806227727406026,-0.9247303957755771);
\draw [line width=0.7pt,color=black] (0.3806227727406026,-0.9247303957755771)-- (-0.5806227727406027,-0.814172706355976);
\draw [line width=0.7pt,color=black] (-0.5806227727406027,-0.814172706355976)-- (0.3806227727406025,0.924730395775577);
\draw [line width=0.7pt,color=black] (0.3806227727406026,-0.9247303957755771)-- (-0.5806227727406026,0.814172706355976);
\draw [line width=0.7pt,color=black] (-0.5806227727406026,0.814172706355976)-- (0.3806227727406025,0.924730395775577);
\begin{scriptsize}
\draw [fill=black] (0,0) circle (1.0pt);
\draw [fill=black] (-0.5806227727406026,-0.8141727063559762) circle (1.0pt);
\draw [fill=black] (-0.419377227259398,-0.907812062739981) circle (1.0pt);
\draw [fill=black] (-0.41937722725939736,0.9078120627399813) circle (1.0pt);
\draw [fill=black] (-0.5806227727406019,0.8141727063559766) circle (1.0pt);
\draw [fill=black] (-0.21937722725939746,0.9756401140584465) circle (1.0pt);
\draw [fill=black] (-0.21937722725939746,-0.9756401140584465) circle (1.0pt);
\draw [fill=black] (-0.019377227259398083,-0.9998122439056933) circle (1.0pt);
\draw [fill=black] (-0.01937722725939739,0.9998122439056933) circle (1.0pt);
\draw [fill=black] (0.18062277274060312,-0.9835524459669126) circle (1.0pt);
\draw [fill=black] (0.18062277274060257,0.9835524459669127) circle (1.0pt);
\draw [fill=black] (0.3806227727406026,-0.9247303957755771) circle (1.0pt);
\draw [fill=black] (0.3806227727406025,0.924730395775577) circle (1.0pt);
\draw[color=black] (1,2) node {$\mathcal{F}$};
\draw[color=black] (-1,0.5) node {$\mathcal{D}$};
\draw[color=black] (0.1,0.1) node {$F$};
\end{scriptsize}
\end{tikzpicture}
\caption{An illustration of $4$-isoperiodicity.}\label{fig.7}
\end{figure}
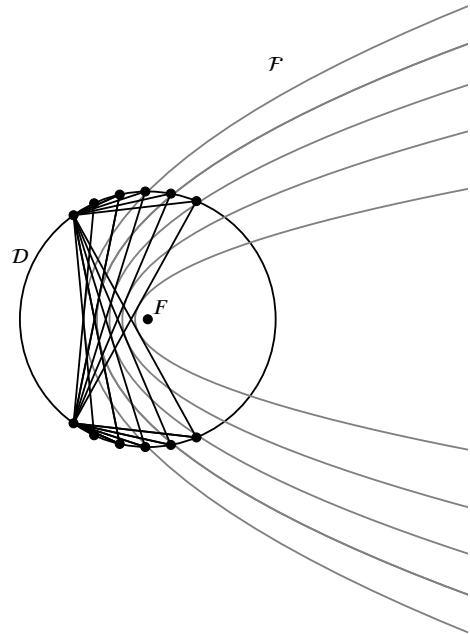

\begin{corollary}\label{cor.3.2} A circle and a parabola do not form an $n$-Poncelet pair $(\mathcal{D}(E),\mathcal{P})$ for $n\neq 4$ and for any $\mathcal{P}\in \mathcal{F}$ if $E=F$. 
\end{corollary}

Denote
\begin{equation}\label{eq.3.2.3}
\Sigma = S^1 \cup \{(0,0)\}.
\end{equation}
\begin{corollary}\label{cor.3.3}
If $(\mathcal{D}(E),\mathcal{P}(p))$ is an $n$-Poncelet pair for $n\neq 3,4$, then $E \in \mathbb{R}^2\setminus \Sigma$.
\end{corollary}

\begin{corollary}\label{cor.3.4}
If $\mathcal{F}$ is an $n$-isoperiodic family of parabolas with $\mathcal{D}(E)$ for $n\neq 3,4$, then $x_E=0$ and $y_E \in \mathbb{R}\setminus \{0,\pm 1\}$.
\end{corollary}
\begin{proof}
By Corollary \ref{cor.3.3}, $E=(x_E,y_E)\in \mathbb{R}^2\setminus\Sigma$. Now if $x_E \neq 0$, then by Theorem \ref{thm.3.4}, there exists a unique $\mathcal{P}\in \mathcal{F}$ such that $(\mathcal{D}(E),\mathcal{P})$ is a 4-Poncelet pair, so it cannot be $n$-isoperiodic with $\mathcal{D}(E)$ for $n\neq 4$. Hence, $x_E=0$. Therefore, $y_E \notin \{0,\pm 1\}$.
\end{proof}

\subsubsection{Geometric Approach to $4$-isoperiodicity}

\begin{theorem}\label{thm.3.6}
Given a parabola $\mathcal{P}$ with the focus $F$ and a circle $\mathcal{D}(F)$. Then, $(\mathcal{D}(F), \mathcal{P})$ is a $4$-Poncelet pair. Moreover, $(\mathcal{D}(F), \mathcal{P}')$ is a $4$-Poncelet pair for every parabola $\mathcal{P}'$ confocal with $\mathcal{P}$.
\end{theorem}
\begin{proof}
\begin{figure}
    \centering
\definecolor{qqqqff}{rgb}{0.,0.,1.}
\definecolor{uuuuuu}{rgb}{0.26666666666666666,0.26666666666666666,0.26666666666666666}
\definecolor{xdxdff}{rgb}{0.49019607843137253,0.49019607843137253,1.}
\definecolor{ududff}{rgb}{0.30196078431372547,0.30196078431372547,1.}
\begin{tikzpicture}[scale=2]
\clip(-1.5,-1.5) rectangle (1.5,1.5);
\draw [shift={(-1.,0.)},line width=0.5pt] (0,0) -- (-30.:0.20765031588921398) arc (-30.:0.:0.20765031588921398) -- cycle;
\draw [shift={(0.5,-0.8660254037844386)},line width=0.5pt] (0,0) -- (120.:0.20765031588921398) arc (120.:150.:0.20765031588921398) -- cycle;
\draw [shift={(-1.,0.)},line width=0.5pt] (0,0) -- (-60.:0.20765031588921398) arc (-60.:-30.:0.20765031588921398) -- cycle;
\draw [shift={(0.5,-0.8660254037844386)},line width=0.5pt] (0,0) -- (150.:0.20765031588921398) arc (150.:180.:0.20765031588921398) -- cycle;
\draw [shift={(-1.,0.)},line width=0.5pt] (-30.:0.20765031588921398) arc (-30.:0.:0.20765031588921398);
\draw[line width=0.5pt] (-0.8144683072671108,-0.04971306723815671) -- (-0.7843820868239397,-0.05777464570920912);
\draw [shift={(0.5,-0.8660254037844386)},line width=0.5pt] (120.:0.20765031588921398) arc (120.:150.:0.20765031588921398);
\draw[line width=0.5pt] (0.3641813745052674,-0.7302067782897061) -- (0.34215673253314866,-0.7081821363175874);
\draw [shift={(-1.,0.)},line width=0.5pt] (-60.:0.20765031588921398) arc (-60.:-30.:0.20765031588921398);
\draw[line width=0.5pt] (-0.8572593115703206,-0.1285242931479413) -- (-0.8341121729060481,-0.14936607041517516);
\draw[line width=0.5pt] (-0.8714757068520587,-0.1427406884296794) -- (-0.8506339295848251,-0.1658878270939518);
\draw [shift={(0.5,-0.8660254037844386)},line width=0.5pt] (150.:0.20765031588921398) arc (150.:180.:0.20765031588921398);
\draw[line width=0.5pt] (0.3121207910922474,-0.8260904451297231) -- (0.28165389235044963,-0.8196145058884176);
\draw[line width=0.5pt] (0.31732435291560485,-0.8066704880246275) -- (0.2877012750100274,-0.7970453665500635);
\draw [line width=1.pt] (0.,0.) circle (1.cm);
\draw [samples=100,rotate around={-90.:(-0.25,0.)},xshift=-0.25cm,yshift=0.cm,line width=1.pt,domain=-4.0:4.0)] plot (\x,{(\x)^2/2/0.5});
\draw [line width=1.pt,color=black] (-1.,0.)-- (0.,0.);
\draw [line width=1.pt,color=black] (0.,0.)-- (0.5,-0.8660254037844386);
\draw [line width=1.pt] (-0.5,-2.953924008264976) -- (-0.5,3.560314787585157);
\draw [line width=1.pt,color=black] (-1.,0.)-- (-0.5,-0.8660254037844386);
\draw [line width=1.pt,color=black] (-0.5,-0.8660254037844386)-- (0.5,-0.8660254037844386);
\draw [line width=1.pt,color=qqqqff] (-1.,0.)-- (0.5,-0.8660254037844386);
\draw [line width=1.pt,color=qqqqff] (-1.,0.)-- (0.5,0.8660254037844386);
\begin{scriptsize}
\draw[color=black] (-0.8,0.8) node {$\mathcal{D}$};
\draw[color=black] (-0.4,1.3) node {$\ell$};
\draw[color=black] (1.3,1.1) node {$\mathcal{P}$};
\draw [fill=black] (0.,0.) circle (1.0pt);
\draw[color=black] (0.06,0.14) node {$F$};
\draw [fill=black] (-1.,0.) circle (1.0pt);
\draw[color=black] (-1.15,0.1) node {$A$};
\draw [fill=black] (0.5,0.8660254037844386) circle (1.0pt);
\draw[color=black] (0.47869088091291345,1.0335573837588032) node {$B$};
\draw [fill=black] (0.5,-0.8660254037844386) circle (1.0pt);
\draw[color=black] (0.46,-1) node {$D$};
\draw [fill=black] (-0.5,-0.8660254037844386) circle (1.0pt);
\draw[color=black] (-0.6,-1.) node {$G$};
\end{scriptsize}
\end{tikzpicture}
    \caption{A geometric proof of $4$-isoperiodicity.}
    \label{fig.8}
\end{figure}
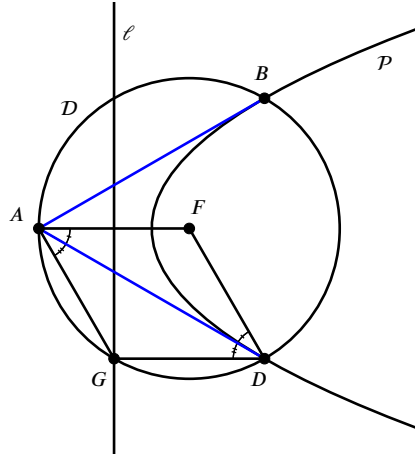
Consider a circle $\mathcal{D}$ with the center at $F$ as shown in the Figure \ref{fig.8}. Denote its intersection points with the parabola $\mathcal{P}$ by $B$ and $D$. Denote its intersection point with the axis of the parabola which lies outside of the parabola by $A$. Denote by $H$ the orthogonal projection of $D$ to the directrix $\ell$. Then $\parallelogram AGDF$ is a rhombus by the focal property (i) of Lemma \ref{lemm.3.1}. Thus, its diagonal $AD$ is the bisector of the angle $GDF$. Thus, it follows from the focal property (ii) of Lemma \ref{lemm.3.1} that $AD$ is tangent to the parabola $\mathcal{P}$ at $D$. 

Thus, $ABAD$ is a $4$-Poncelet polygon inscribed in $\mathcal{D}(F)$ and circumscribed about $\mathcal{P}$. By the Poncelet theorem, this is true for all polygonal lines that start from any point of $\mathcal{D}$ other than $A$ as well.

The same holds for any parabola $\mathcal{P}'$ confocal to $\mathcal{P}$.
\end{proof}

Motivated by the Theorems \ref{thm.3.2} and \ref{thm.3.5}, we now ask the following question:\\
    \textit{Does there exist $n \geq 3$ such that the confocal family $\mathcal{F}$ is $n$-isoperiodic with a circle $\mathcal{D}(E)$ for some $E\in \mathbb{R}^2\setminus\Sigma$ and $n \neq 3,4$?}

We will answer this question in the next section, see Theorem \ref{thm.4.5}.

\section{Cyclic Pentagons, Hexagons, and Heptagons circumscribing Parabolas from a Confocal Family}\label{sec.4}
\subsection{Pentagons}\label{sec.4.1} 
For $n=5$, we have the following (recalling $\Sigma$ from eq. \eqref{eq.3.2.3}):
\begin{proposition}\label{prop.4.1} A circle $\mathcal{D}(E)$ and a parabola $\mathcal{P}(p)\in \mathcal{F}$ form a $5$-Poncelet pair $(\mathcal{D}(E),\mathcal{P}(p))$ if and only if $E\in \mathbb{R}^2\setminus \Sigma$ and
\begin{eqnarray}\label{eq.4.1.1}
\mathcal{Q}^5_2 (x_E,y_E,p)&:=&4(x_E^2 + y_E^2) p^2 + 4x_E (x_E^2 + y_E^2 - 1)p -(x_E^2 + y_E^2 - 1)^3=0.
\end{eqnarray}
\end{proposition}
\begin{proof} We calculate the determinant:
    \begin{eqnarray}\label{eq.4.1.2}
\left|\begin{array}{ccc}
    A_{2}(x_E,y_E,p) & A_{3}(x_E,y_E,p)\\
     A_{3}(x_E,y_E,p) & A_{4}(x_E,y_E,p)
\end{array}\right| =\frac{1}{16p^4}\mathcal{Q}^5_2 (x_E,y_E,p).
\end{eqnarray}
Therefore, the condition $\mathcal{Q}^5_2 (x_E,y_E,p)=0$ is equivalent to the Cayley condition eq. \eqref{eq.2.0.1} for $n=5$.
\end{proof}

\begin{theorem}\label{thm.4.1}
A confocal family of parabolas is not $5$-isoperiodic with a circle.
\end{theorem}
\begin{proof}
The leading coefficient of the polynomial $\mathcal{Q}^5_2 (x_E,y_E,p)\in \mathbb{R}[p]$ is equal to 0 if and only if $(x_E,y_E)=(0,0)$. So, $\mathcal{Q}^5_2 (x_E,y_E,p)$ is not the zero polynomial in $\mathbb{R}[p]$ for any $(x_E,y_E)\in \mathbb{R}^2\setminus \Sigma$, thus, cannot be satisfied by every $p\in \mathbb{R}^*$. 

Alternatively, for every $(x_E,y_E)\in \mathbb{R}^2\setminus\Sigma$, the polynomial $\mathcal{Q}^5_2 (x_E,y_E,p)\in \mathbb{R}[p]$ is quadratic in $p$, can have at most two zeros. The proof is complete. 
\end{proof}

The polynomial $\mathcal{Q}^5_2 (x,y,p)$ is quadratic in $p$ and its discriminant is given by
\begin{eqnarray}\label{eq.4.1.3}
\mathrm{Disc}_p(\mathcal{Q}^5_2)(x,y)&=&16(x^2 + y^2-1)^2 (x^2 + y^2 - y) (x^2 + y^2 + y).
\end{eqnarray}

Let $Q(x,y)\in \mathbb{Q}[x,y]$ be a polynomial and $\Gamma$ be the following algebraic curve defined by $Q(x,y)$, that is, 
\begin{eqnarray*}
\Gamma&=& \left\{(x,y)\in \mathbb{R}^2 \mid Q(x,y)=0\right\}.
\end{eqnarray*}
In what follows, we will use the following notations:
\begin{eqnarray*}
\Gamma_{+} &=& \left\{(x,y)\in \mathbb{R}^2 \mid Q(x,y) > 0 \right\},\\
\Gamma_{-} &=& \left\{(x,y)\in \mathbb{R}^2 \mid Q(x,y) < 0 \right\}.
\end{eqnarray*}

We also define 
\begin{eqnarray*}
\Gamma^{5}= \left\{(x,y)\in \mathbb{R}^2~|~x^2+\left(y \pm \frac{1}{2}\right)^2 = \frac{1}{4}\right\}.
\end{eqnarray*}

Based on the discriminant, we have the following:

\begin{proposition}\label{prop.4.2}
Given a confocal family of parabolas $\mathcal F$, there exist(s)
\begin{itemize}
\item[a)] two $5$-Poncelet pairs $(\mathcal{D}(E),\mathcal{P}(p_{\pm}))$ corresponding to the distinct real zeros of $\mathcal{Q}^5_2(x_E,y_E,p)$ if $E\in \Gamma ^5_{+}\setminus S^1$; the zeros are given by
\begin{eqnarray} \label{eq.4.1.4}
p_{\pm}(E) =\frac{\left(-x_E\pm  \sqrt{\left(x_E^2+y_E^2\right)^2-y_E^2}\right)(x_E^2+y_E^2-1)}{2(x_E^2 + y_E^2)}.
\end{eqnarray}

\item[b)] one $5$-Poncelet pair $(\mathcal{D}(E),\mathcal{P}(p))$ if $E \in \Gamma^{5}\setminus\Sigma$ where
\begin{eqnarray}\label{eq.4.1.5}
p(E)=p_{+}(E)=p_{-}(E)=\frac{\sqrt{y_E}\left(1-y_E\right)^{\frac{3}{2}}}{2y_E},~0<y_E<1;
\end{eqnarray}
\item[c)] no real $5$-Poncelet pairs if $E\in \Gamma^5_{-}$.
\end{itemize}
\end{proposition}
\begin{proof}
The proofs of a) and c) go as follows. 

Take $E \in \mathbb{R}^2\setminus \Sigma$. Then for every $E \in \Gamma^{5}_{+}\setminus S^1$, $\mathrm{Disc}_p(\mathcal{Q}^5_2)(x_E,y_E)>0$. The zeros of $\mathcal{Q}^5_2(x_E,y_E,p)$ in $p$ are given by the eq. \eqref{eq.4.1.4}. See Figure \ref{fig.9}.

Similarly, for every $E \in \Gamma^{5}_{-}$, $\mathrm{Disc}_p(\mathcal{Q}^5_2)(x_E,y_E)<0$, so the zeros are complex conjugates.

For every $E \in \Gamma_{5}$, $\mathrm{Disc}_p(\mathcal{Q}^5_2)(x_E,y_E)=0$. The zeros of $\mathcal{Q}^5_2(x_E,y_E,p)$ in $p$ are given by the eq. \eqref{eq.4.1.5}. This proves b).

Furthermore, it can be easily shown that
\begin{itemize}
\item for every $E \in S^1_{+}$,
\begin{itemize}
    \item[] $p_{+}(E)>0>p_{-}(E)$ if $x_E \neq 0$;
    \item[] $p_{-}(E)=-p_{+}(E)$ if $x_E=0$;
\end{itemize}
\item for every $E \in \Gamma^{5}_{+} \cap S^1_{-}$,
\begin{equation*}
    p_{-}(E)>p_{+}(E)>0\quad \textrm{if} ~x_E\neq 0;
\end{equation*} 
\item for every $E \in \Gamma^{5}\setminus\Sigma$,
\begin{equation*}
    \mathrm{sign}(p(E)=\mathrm{sign}(x_E).
\end{equation*}
\end{itemize}
\end{proof}

\begin{figure}
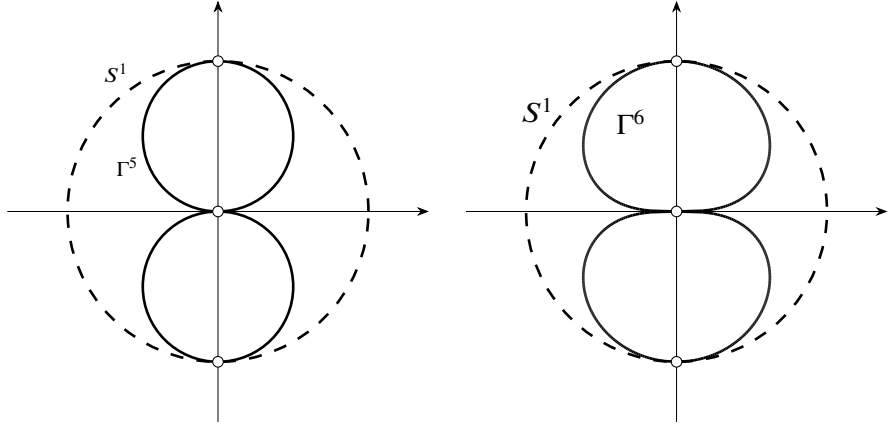

    \centering

\caption{$\Gamma^5$ and $\Gamma^6$.}\label{fig.9}
\end{figure}

\begin{corollary}\label{cor.4.1}
If $\mathcal{F}$ is an $n$-isoperiodic family of parabolas with $\mathcal{D}(E)$ for $n\neq 3,4$, then $x_E=0$ and $y_E \in (-1,0)\cup (0,1)$.
\end{corollary}

\subsection{Hexagons }\label{sec.4.2}
For $n=6$, we have:
\begin{proposition}\label{prop.4.3} A circle $\mathcal{D}(E)$  and a parabola $\mathcal{P}(p)\in \mathcal{F}$ form a $6$-Poncelet pair $(\mathcal{D}(E),\mathcal{P}(p))$ if and only if $E\in \mathbb{R}^2\setminus \Sigma$ and
\begin{eqnarray}\label{eq.4.2.1}
\mathcal{Q}^6_2 (x_E,y_E,p)&:=&4(x_E^2 + y_E^2) (x_E^2 + y_E^2 + 1) p^2 + 4 x_E (2x_E^2+2y_E^2 + 1) (x_E^2 + y_E^2 - 1) p\nonumber\\
&& + (3x_E^2 - y_E^2 + 1) (x_E^2 + y_E^2 - 1)^2=0.
\end{eqnarray}
\end{proposition}
\begin{proof} For $n=6$, we obtain
\begin{eqnarray}\label{eq.4.2.2}
\left| \begin{matrix}
A_3(x_E,y_E,p) & A_4(x_E,y_E,p)\\
A_4(x_E,y_E,p) & A_5(x_E,y_E,p)
\end{matrix}\right|=-\frac{1}{64p^6}\mathcal{Q}^3_0 (x_E,y_E,p)\mathcal{Q}^6_2 (x_E,y_E,p).
\end{eqnarray}
 
Since for any $E\in \mathbb{R}^2\setminus \Sigma$, $\mathcal{Q}^3_0 (x_E,y_E,p) \neq 0$, so we must have $\mathcal{Q}^6_2 (x_E,y_E,p)=0$ for $(\mathcal{D}(E),\mathcal{P}(p))$ to be a 6-poncelet pair.
\end{proof}

\begin{theorem}\label{thm.4.2}
A confocal family of parabolas is not $6$-isoperiodic with a circle.
\end{theorem}
\begin{proof}
This follows from the eq. \eqref{eq.4.2.1} in Proposition \ref{prop.4.3}, in a manner similar to the proof of Theorem \ref{thm.4.1}.
\end{proof}

For every $(x,y)\in \mathbb{R}^2\setminus \Sigma$, $\mathcal{Q}^6_2 (x,y,p)\in \mathbb{R}[p]$ is quadratic in $p$; its discriminant in $p$ is given by
\begin{equation}\label{eq.4.2.3}
\mathrm{Disc}_p(\mathcal{Q}^6_2)(x,y)=16(x^2 + y^2 - 1)^2 \left((x^2 + y^2)^3 - y^2\right).
\end{equation}

Let $\Gamma^6, \Lambda^6$ denote the following sets:
\begin{eqnarray*}
\Gamma^6&=&\{(x,y) \in \mathbb{R}^2~|~(x^2+y^2)^3-y^2=0\},\\
\Lambda^6&=&\{(x,y) \in \mathbb{R}^2~|~3x^2-y^2+1=0\}.
\end{eqnarray*}

The following proposition finds the number of $6$-Poncelet pairs by specifying the nature of zeros of $\mathcal{Q}^6_2(x_E,y_E,p)\in \mathbb{R}[p]$ for  $\mathbb{R}^2\setminus\Sigma$. 

\begin{proposition}\label{prop.4.4}
Given a confocal family of parabolas $\mathcal F$, for every $E\in \mathbb{R}^2\setminus \Sigma$, there exist(s)
\begin{itemize}
    \item [a)] two $6$-Poncelet pairs $(\mathcal{D}(E),\mathcal{P}(p_{\pm}))$ corresponding to two distinct real zeros $p_{\pm}$ if $E\in \Gamma^6_{+}\setminus \Sigma$; the zeros are given by
    \begin{eqnarray}\label{eq.4.2.4}
p_{\pm}(E)=\frac{\left(-x_E (2x_E^2 + 2y_E^2 + 1) \pm  \sqrt{(x_E^2 + y_E^2)^3 - y_E^2}\right)(x_E^2 + y_E^2 - 1)}{2(x_E^2+y_E^2)(x_E^2+y_E^2+1)}.
\end{eqnarray}
    \item [b)] one $6$-Poncelet pair $(\mathcal{D}(E),\mathcal{P}(p))$ corresponding to the real zero with multiplicity 2 if $E\in \Gamma^6\setminus\Sigma$ satisfying 
    \begin{equation*}
        p(E)=p_{+}(E)=p_{-}(E),\quad \mathrm{sign}(p_{\pm})=\mathrm{sign}(x_E);
    \end{equation*}
    and the zero is given by
    \begin{eqnarray}\label{eq.4.2.5}
p(E)=-\frac{\sqrt{1-y_E^{\frac{4}{3}}} \left(2y_E^{\frac{2}{3}} + 1\right) \left(y_E^{\frac{2}{3}} - 1\right)}{2y_E^{\frac{1}{3}}\left(y_E^{\frac{2}{3}}+1\right)},~ y_E\in (0,1);
\end{eqnarray}
\item[c)] no real $6$-Poncelet pair if $E\in \Gamma^6_{-}$.
\end{itemize}
\end{proposition}
\begin{proof} 
Since $\mathrm{Disc}_p(\mathcal{Q}^6_2)(x_E,y_E)>0$ if $E\in \Gamma^6_{+}\setminus S^1$, and $\mathrm{Disc}_p(\mathcal{Q}^6_2)(x_E,y_E)<0$ if $E\in \Gamma^6_{-}$, these prove a) and c), respectively.

Similarly, $\mathrm{Disc}_p(\mathcal{Q}^6_2)(x_E,y_E)=0$ if $E\in \Gamma^6$. This proves b).
Furthermore,
\begin{itemize}
    \item for every $E \in S^1_{+}\cap \Lambda^6_{+}$
    \begin{equation*}
        \mathrm{sign}(p_{\pm}(E))=-\mathrm{sign}(x_E);
    \end{equation*}
       \item for every $E \in S^1_{+}\cap \Lambda^6_{-}$
    \begin{eqnarray*}
        p_{+}(E)>0>p_{-}(E) \quad \textrm{if} \quad x_E \neq 0\\
        p_{-}(E)=-p_{+}(E) \quad \textrm{if} \quad x_E = 0
    \end{eqnarray*}
   \item for every $E \in S^1_{+}\cap \Lambda^6$
    \begin{eqnarray*}
        p_{+}(E)>0,~p_{-}(E)=0 \quad \textrm{if} \quad x_E < 0\\
        p_{+}(E)=0,~p_{-}(E)<0 \quad \textrm{if} \quad x_E > 0
    \end{eqnarray*}
   \item for every $E \in \Gamma^6_{+}\cap \Lambda^6_{+}$,
    \begin{eqnarray*}
        \mathrm{sign}(p_{\pm}(E))=\mathrm{sign}(x_E);
    \end{eqnarray*}
    \item for every $E\in \Gamma^6\setminus\Sigma$; 
    \begin{equation*}
        p(E)=p_{+}(E)=p_{-}(E),\quad \mathrm{sign}(p_{\pm}(E))=\mathrm{sign}(x_E).
    \end{equation*}
\end{itemize}
See Figure \ref{fig.9}.
\end{proof}

\subsection{Heptagons. Are there confocal families of parabolas $n$-isoperiodic with a circle for $n\ne 3, 4$?}\label{sec.4.3}

Cayley's conditions for $n=7$ leads to quartic polynomials in $p$. We will first review the nature of zeros of a quartic polynomial.
\subsubsection{Nature of zeros of a general quartic polynomial based on its discriminant and associated quantities}\label{app.B}
\noindent
For the general quartic polynomial $f(x)\in \mathbb{R}[x]$:
\begin{equation}\label{eq.4.3.1}
    f(x)=A x^4+B x^3+C x^2+Dx+E,
\end{equation}
the discriminant is given by (see e.g., \cite{Irving2013}):
\begin{eqnarray*}
    \mathrm{Disc}_x(f)&=&B^2 C^2 D^2 - 4 A C^3 D^2 - 4 B^3 D^3 + 18 A B C D^3 - 27 A^2 D^4 - 
 4 B^2 C^3 E + 16 A C^4 E\\
 &&+ 18 B^3 C D E - 80 A B C^2 D E - 
 6 A B^2 D^2 E + 144 A^2 C D^2 E - 27 B^4 E^2\\
 &&+ 144 A B^2 C E^2 - 
 128 A^2 C^2 E^2 - 192 A^2 B D E^2 + 256 A^3 E^3.
\end{eqnarray*}
Define
\begin{eqnarray}
P_f&:=&8A C -3B^2 \label{eq.4.3.2},\\
D_f&:=&-3 B^4 - 16A^2 C^2 + 64A^3 E + 16A B^2 C - 16A^2 B D \label{eq.4.3.3},\\
R_f&:=& B^3 + 8A^2 D - 4A B C \label{eq.4.3.4}\\
O_f&:=& C^2 + 12A E - 3B D \label{eq.4.3.5}.
\end{eqnarray}
Then the complete classification of the nature of zeros of a general quartic in eq. \eqref{eq.4.3.1} is given by the following theorem of Rees \cite{Rees1922} (also see \cite{Dragovic-Murad2025b}):
\begin{theorem}[Rees 1922 \cite{Rees1922}]\label{thm.4.3}
Let $f(x)\in \mathbb{R}[x]$ be the general quartic polynomial in eq. \eqref{eq.4.3.1}.
\begin{itemize}
\item[a)] If $\mathrm{Disc}_x(f) < 0$, then  $f$ has two real simple and two complex conjugate zeros.
\item[b)] If $\mathrm{Disc}_x(f) > 0$, then either $f$ has four real simple or two pairs of complex conjugate zeros.
\begin{itemize}
\item[b.1)] If $P_f < 0 \wedge D_f<0$, then all four zeros are real and simple.
\item[b.2)] If $P_f \geq 0 \vee \left(P_f < 0 \wedge D_f > 0\right)$, then there are two pairs of complex conjugate zeros.
\end{itemize}
\item[c)] If $\mathrm{Disc}_x(f) = 0$, then (and only then) $f$ has a multiple zero. The following different cases can occur:
\begin{itemize}
\item[c.1)] If $P_f < 0 \wedge D_f < 0 \wedge O_f > 0$, then there are a real double and two real simple zeros.
\item[c.2)] If $\left(P_f \leq  0 \wedge D_f>0\right) \vee \left(P_f>0 \wedge O_f>0 \wedge (D_f<0 \vee R_f \neq 0)\right)$, then there are a real double zero and two complex conjugate zeros.
\item[c.3)] If $P_f<0 \wedge O_f = 0$, then there are a real triple zero and a real simple zero.
\item[c.4)] If $D_f = 0$, then:
\begin{itemize}
\item[c.4.1)] If $P_f < 0$, then there are two real double zeros.
\item[c.4.2)] If $P_f > 0 \wedge R_f = 0$, then there is a pair of double complex conjugate zeros.
\item[c.4.3)] If $O_f = 0$, then there is a real zero, equal to $-B/4$, with multiplicity four.
\end{itemize}
\end{itemize}
\end{itemize}
\end{theorem}

\subsubsection{Cayley's Condition for $n=7$}
From Theorem \ref{thm.2.1}, we have:
\begin{proposition}\label{prop.4.5} A circle $\mathcal{D}(E)$ and a parabola $\mathcal{P}(p)\in \mathcal{F}$ form a $7$-Poncelet pair $(\mathcal{D}(E),\mathcal{P}(p))$ if and only if $E\in \mathbb{R}^2\setminus \Sigma$ and
\begin{eqnarray}\label{eq.4.3.6}
\mathcal{Q}^7_4(x_E,y_E,p)&:=&16(x_E^2 + y_E^2)^3 p^4 +48x_E (x_E^2 + y_E^2)^2 (x_E^2 + y_E^2 - 1) p^3 \nonumber\\
&& +4(x_E^2 + y_E^2) (x_E^2 + y_E^2 - 1)^2 (13x_E^2 + y_E^2 - 1) p^2\nonumber\\
&&+4x_E (x_E^2 + y_E^2 - 1)^3 (5x_E^2 + y_E^2 - 1) p - (x_E^2 + y_E^2 - 1)^6=0.
\end{eqnarray}
\end{proposition}
\begin{proof}
Follows from the Cayley condition eq. \eqref{eq.2.0.1} for $n=7$ and the following equation:
\begin{eqnarray}\label{eq.4.3.7}
\left|\begin{array}{ccc}
    A_2(x_E,y_E,p) & A_3(x_E,y_E,p) & A_4(x_E,y_E,p)\\
    A_3(x_E,y_E,p) & A_4(x_E,y_E,p) & A_5(x_E,y_E,p) \\
    A_4(x_E,y_E,p) & A_5(x_E,y_E,p) & A_6(x_E,y_E,p)
\end{array}\right| = -\frac{\sqrt{-p^2}}{512p^{10}}\mathcal{Q}^7_4(x_E,y_E,p).
\end{eqnarray}
\end{proof}

\begin{theorem}\label{thm.4.4}
A confocal family of parabolas is not $7$-isoperiodic with a circle.
\end{theorem}
\begin{proof}
For every $(x_E,y_E)\in \mathbb{R}^2\setminus\Sigma$, the polynomial $\mathcal{Q}^7_4 (x_E,y_E,p)\in \mathbb{R}[p]$ is quartic in $p$, can have at most four zeros. The proof is complete. 
\end{proof}

For every $(x,y)\in \mathbb{R}^2\setminus \Sigma$, the discriminant and the associated quantities in eqs. \eqref{eq.4.3.2}-\eqref{eq.4.3.5} for the quartic polynomial $\mathcal{Q}^7_4(x,y,p)\in \mathbb{R}[p]$ are calculated as:
\begin{eqnarray}
\mathrm{Disc}_p(\mathcal{Q}^7_4)(x,y)&=&-65536(x^2 + y^2)^6 (x^2 + y^2 - 1)^{15} \Psi_1(x,y);\label{eq.4.3.8}\\
P_{\mathcal{Q}^7_4}(x,y)&=&-256 (x^2 + y^2)^4 (x^2 + y^2 - 1)^2 \Psi_2(x,y);\label{eq.4.3.9}\\
D_{\mathcal{Q}^7_4}(x,y)&=&-65536 (x^2 + y^2)^8 (x^2 + y^2 - 1)^4 \Psi_3(x,y);\label{eq.4.3.10}\\
O_{\mathcal{Q}^7_4}(x,y)&=& -16 (x^2 + y^2)^2 (x^2 + y^2 - 1)^5 \Psi_4(x,y) \label{eq.4.3.11}\\
R_{\mathcal{Q}^7_4}(x,y)&=& -4096 x (x^2 + y^2)^6 (x^2 + y^2 - 1)^3\Psi_5(x,y)\label{eq.4.3.12}
\end{eqnarray}
where
\begin{eqnarray*}
\Psi_1(x,y)&=&16\left(x^2+y^2\right)^6 - x^{10} - 71x^8 y^2 + x^8 - 247x^6 y^4 + 43x^6 y^2 - 325x^4 y^6 + 108x^4 y^4\\
&& - 23x^4 y^2 - 188x^2 y^8 + 91x^2 y^6 - 2x^2 y^4+ 3x^2 y^2 - 40y^{10} + 25y^8 + 5y^6 - 5y^4 - y^2,\\
\Psi_2(x,y)&=& x^2 - 2y^2 + 2,\\
\Psi_3(x,y)&=&4(x^2 + y^2)^3 -7(x^2 + y^2)^2 + 2 (x^2 + y^2) + 3 x^4 + 1,\\
\Psi_4(x,y)&=&12(x^2 + y^2)^2 - 13(x^2+y^2) +12y^2 + 1,\\
\Psi_5(x,y)&=&2x^2 + y^2 - 1.
\end{eqnarray*}

Let $\mathcal{R}^j, j=1,2,\ldots, 5$ denote the following sets:
\begin{eqnarray*}
    \mathcal{R}^j&=& \left\{(x,y)\in \mathbb{R}^2 \mid \Psi_j(x,y)=0\right\}.
\end{eqnarray*}

We further observe that
\begin{equation*}
     S^1_{-} \cap \mathcal{R}^1_{+} \neq \emptyset, \quad S^1_{+} \cap \mathcal{R}^1_{-} =\emptyset, \quad S^1_{-} \cap \mathcal{R}^1_{-} =\mathcal{R}^1_{-}, \quad  S^1_{+} \cap \mathcal{R}^1_{+} =S^1_{+}.
\end{equation*}

\begin{lemma} \label{lemm.4.1}
For every $E\in \mathbb{R}^2\setminus \Sigma$,
\begin{itemize}
    \item[a)] $D_{\mathcal{Q}^7_4}(x_E,y_E)<0$;
    \item[b)] if $\mathrm{Disc}_p(\mathcal{Q}^7_4)(x_E,y_E)\geq 0$ then $P_{\mathcal{Q}^7_4}(x_E,y_E)<0$.
\end{itemize}
    
\end{lemma}
\begin{proof} To prove a), it is sufficient to show that $\Psi_3(x_E,y_E)>0$ for every $(x_E,y_E) \in \mathbb{R}^2\setminus \Sigma$.

To show this, we consider the following polynomial $f(X,c)\in \mathbb{R}[X]$:
\begin{equation*}
    f(X,c)=4X^3-7X^2+2X+c, \quad c \geq 1.
\end{equation*}
Clearly, for all $X > 0$,
\begin{equation*}
    f(X,1)=(X - 1)^2 (4X + 1) \geq 0.
\end{equation*}
Therefore, $f(X,1)>0$ for all $X>0$ but $X \neq 1$. Also
\begin{equation*}
    f(X,c) = f(X,1) + c-1 > 0.
\end{equation*}
for all $c \geq 1$. Now, to finish the proof, set $X=x_E^2+y_E^2$ and $c=3x_E^4+1$. Note that $c\geq 1$ and $X>0$ but $X\neq 1$ for every $(x_E,y_E)\in \mathbb{R}\setminus \Sigma$. 
Thus, 
\begin{equation*}
 \Psi_3(x_E,y_E)=f(X,c) > 0.  
\end{equation*} 

b) If $\mathrm{Disc}_p(\mathcal{Q}^7_4)(x_E,y_E)\geq 0$ then $0<x_E^2+y_E^2<1$ and $x_E \neq 0$. It then follows:
\begin{equation*}
  \Psi_2(x_E,y_E) = x_E^2-2y_E^2+2>3x_E^2>0.
\end{equation*}
The proof is complete.
 \end{proof}

\begin{corollary}\label{cor.4.2}
For every $E\in \mathbb{R}^2\setminus\Sigma$, the polynomial $\mathcal{Q}^7_4(x_E,y_E,p)\in \mathbb{R}[p]$ has at most one pair of complex conjugate zeros of multiplicity 1.   
\end{corollary}
\begin{proof} 
Since $D_{\mathcal{Q}^7_4}(x_E,y_E) < 0$ for every $(x_E,y_E)\in \mathbb{R}^2\setminus \Sigma$, by Theorem \ref{thm.4.3} the polynomial $\mathcal{Q}^7_4(x_E,y_E,p)$ has two distinct pairs of complex conjugate zeros if and only if 
\begin{equation*}
    \mathrm{Disc}_p(\mathcal{Q}^7_4)(x_E,y_E)>0 \wedge P_{\mathcal{Q}^7_4}(x_E,y_E)>0.
\end{equation*}
Now, if $(x_E,y_E)\in \Sigma$, then $\mathrm{Disc}_p(\mathcal{Q}^7_4)(x_E,y_E)=0$. And for every $(x_E,y_E)\in \mathbb{R}^2\setminus\Sigma$, by Lemma \ref{lemm.4.1} b),
\begin{equation*}
    \mathrm{Disc}_p(\mathcal{Q}^7_4)(x_E,y_E)>0 \Rightarrow P_{\mathcal{Q}^7_4}(x_E,y_E)<0.
\end{equation*}

Also since $D_{\mathcal{Q}^7_4}(x_E,y_E)<0$, $Q^7_4(x_E,y_E,p)$ does not have a pair of double complex conjugate zeros.
This concludes the proof.
\end{proof}

\begin{proposition}\label{prop.4.6}
Given a confocal family of parabolas $\mathcal F$, there exist 
\begin{itemize}
    \item [a)] two 7-Poncelet pairs corresponding to two distinct real zeros of $\mathcal{Q}^7_4(x_E,y_E,p)\in \mathbb{R}[p]$ if $E\in \mathcal{R}^1_{-} \cup S^1_{+}$;
    \item [b)] three 7-Poncelet pairs corresponding to one real zero of multiplicity 2 and two real zeros  of multiplicity 1 if $E\in \mathcal{R}^1$;
    \item [c)] four 7-Poncelet pairs corresponding to the four distinct real zeros if $E\in \mathcal{R}^1_{+} \cap S^1_{-}$.
\end{itemize} 
\end{proposition}
\begin{proof}
Take $(x_E,y_E)\in \mathbb{R}^2\setminus \Sigma$. Since $D_{\mathcal{Q}^7_4}(x_E,y_E)<0$ by Lemma \ref{lemm.4.1} a), it follows from the Theorem \ref{thm.4.3} that 
\begin{itemize}
    \item[a)]  $\mathrm{Disc}_p(\mathcal{Q}^7_4)(x_E,y_E)<0$ if $(x_E,y_E)\in \mathcal{R}^1_{-} \cup S^1_{+}$.
    \item[b)] $\mathrm{Disc}_p(\mathcal{Q}^7_4)(x_E,y_E)=0$ if $(x_E,y_E)\in \mathcal{R}^1$. But there exists no $(x_E,y_E)\in \mathbb{R}^2$ such that
    \begin{equation*}
        \mathrm{Disc}_p(\mathcal{Q}^7_4)(x_E,y_E)=0 \wedge P_{\mathcal{Q}^7_4}(x_E,y_E)=0 \quad \text{or} \quad \mathrm{Disc}_p(\mathcal{Q}^7_4)(x_E,y_E)=0 \wedge O_{\mathcal{Q}^7_4}(x_E,y_E)=0.
     \end{equation*}

So, there exists no real zero with multiplicity 2 and two complex conjugate zeros, or no real zeros with multiplicities 1 and 3;   
    \item[c)] Since $D_{\mathcal{Q}^7_4}(x_E,y_E)<0$ by Lemma \ref{lemm.4.1}, for every $E\in \mathcal{R}^1_{+} \cup S^1_{-}$ the following condition 
    \begin{equation*}
        \mathrm{Disc}_p(\mathcal{Q}^7_4)(x_E,y_E)>0 \wedge P_{\mathcal{Q}^7_4}(x_E,y_E) < 0. 
    \end{equation*}
    is satisfied.
\end{itemize}
\end{proof}

\begin{corollary}\label{cor.4.3}

The confocal family of parabolas $\mathcal{F}$ is not $n$-isoperiodic with the circle $\mathcal{D}(E)$ for $n \neq 7$ if $E\in \mathbb{R}^2\setminus \Sigma$.
\end{corollary}

We observe that Cayley's condition for $n=8$ leads again to a quartic polynomial in $p$. Thus, a similar analysis for cyclic octagons circumscribing parabolas from a confocal family can be obtained along the same lines as in $n=7$ case.

\subsubsection{The Answer.}
Now we answer the question raised at the end of previous section:
\begin{center}
    \textit{Is a confocal family of parabolas $n$-isoperiodic with a circle for $n \neq 3,4$?}
\end{center}

\begin{theorem}\label{thm.4.5}
The confocal family of parabolas $\mathcal{F}$ is $n$-isoperiodic with some circle $\mathcal{D}$ if and only if $n \in \{3,4\}$.
\end{theorem}
\begin{proof}
We already proved one (`if') direction in Theorems \ref{thm.3.2} and \ref{thm.3.5}. 

To prove the `only if' direction, we suppose, on the way to the contradiction, that $\mathcal{F}$ is $n$-isoperiodic with $\mathcal{D}(E)$. Then by Corollary \ref{cor.3.3}, $E\in \mathbb{R}^2\setminus \Sigma$. Now it follows from Corollary \ref{cor.4.2} that $n=7$. This, however, contradicts the Theorem \ref{thm.4.4}. The proof is complete.
\end{proof}

Theorem \ref{thm.4.5} can be alternatively proved by using the Corollary \ref{cor.4.1} and Theorem \ref{thm.4.3} as follows.
\begin{proof} If $\mathcal{F}$ is $n$-isoperiodic with some circle $\mathcal{D}(E)$, then by Corollary \ref{cor.4.1}, we have
\begin{equation*}
    x_E=0, \quad y_E \in (-1,0) \cup (0,1).
\end{equation*}
 
 As $\mathcal{Q}^7_4(0,y_E,p)$ is  biquadratic in $p$, its discriminant in $p^2$ is given by
\begin{equation*}
    \mathrm{Disc}_{p^2}(\mathcal{Q}^7_4)(0,y_E)=16y_E^4 (y_E + 1)^6 (y_E - 1)^6 (4y_E^2 + 1) >0, \quad \forall y_E \in \mathbb{R}\setminus\{0,\pm 1\}.
\end{equation*}
If $\alpha, \beta$ are the zeros of the biquadratic, then $\alpha \beta <0$. If, without loss of generality, $\alpha>0$, then the real zeros of $\mathcal{Q}^7_4(0,y_E,p)$ are given by: $p_{\pm}(E)=\pm \sqrt{\alpha}$. Therefore, $n=7$; this, however, contradicts the Theorem \ref{thm.4.3}. The proof is complete.
\end{proof}

\section{An algebro-geometric approach to Painlev\'{e} VI equations}\label{sec.5}
\subsection{Picard solutions and Okamoto transformation}\label{sec.5.1}

The Painlev\'e VI equations form the four-parameter family of remarkable second order ordinary differential equations introduced by Paul Painlev\'e and his school at the beginning of the twentieth century.

The Painlev\'e VI equation reads:
\begin{eqnarray}\label{eq.5.1.1}
	\frac{d^2 y}{dx^2} &=& \frac{1}{2} \left(  \frac{1}{y} + \frac{1}{y-1}+ \frac{1}{y-x}  \right) \left( \frac{dy}{dx} \right)^2 - \left( \frac{1}{x} + \frac{1}{x-1} + \frac{1}{y-x} \right) \frac{dy}{dx}\nonumber\\
	&&+ \frac{y(y-1)(y-x)}{x^2(x-1)^2}\left( \alpha +\beta\frac{x}{y^2} + \gamma\frac{x-1}{(y-1)^2} + \delta\frac{x(x-1)}{(y-x)^2}  \right).
\end{eqnarray}

Here, the parameters are $\alpha, \beta, \gamma, \delta \in \mathbb C$, see for example \cite{Okamoto1987,Hitchin1995,Bolibruch2006,Dragovic-Shramchenko2019} and references therein. We will deal here with two particular equations with parameter values
\begin{equation}\label{eq.5.1.2}
	\alpha=\frac{1}{8}, \qquad \beta=-\frac{1}{8}, \qquad \gamma=\frac{1}{8}, \qquad \delta=\frac{3}{8},
\end{equation}
and
\begin{equation}\label{eq.5.1.3}
	\alpha=0, \qquad \beta=0, \qquad \gamma=0, \qquad \delta=\frac{1}{2}.
\end{equation}

We will use the transformed Weierstrass $\wp$-function with periods $2w_1$ and $2w_2$, defined by the following  equation

\begin{equation}\label{eq.5.1.4}
	\left( {\hat \wp^\prime}(z) \right)^2 = \hat \wp(z)(\hat \wp(z)-1)(\hat \wp(z)-x).
\end{equation}

The explicit solution to the  Painlev\'e VI equation with parameters \eqref{eq.5.1.3} was obtained by Picard \cite{Picard1889}. The formula for the solution is:

\begin{equation}\label{eq.5.1.5}
	y_0 (x) = \hat \wp(2c_1 w_1(x) + 2c_2 w_2(x)).
\end{equation}

The transformation
\begin{equation}\label{eq.5.1.6}
	y(x) = y_0 + \frac{y_0(y_0-1)(y_0-x)}{x(x-1)y_0^\prime - y_0(y_0-1)},
\end{equation}
was written in this form in \cite{Dragovic-Shramchenko2019} and follows from \cite{Okamoto1987}. It relates the Picard solution $y_0(x)$, eq. \eqref{eq.5.1.5}, to the general solution $y(x)$ of the
Painlev\'e VI equation with constants \eqref{eq.5.1.2}.

In order to apply the Okamoto transformation, one needs an expression for $y_0'$. Such an expression is obtained in \cite{Dragovic-Shramchenko2019}.

Consider the elliptic curve $\mathcal{E}$ defined by the equation

\begin{equation}\label{eq.5.1.7}
\mathcal{E}: v^2=u(u-1)(u-x).
\end{equation}
It is parameterized by $(u, v)=(\hat \wp, \hat \wp')$. Its Abel map
\begin{equation*}
	\mathcal{A}(p)=\int_{p_\infty}^p\frac{du}{v},  \quad p\in \mathcal{E},
\end{equation*}
maps $\mathcal{E}$ into the Jacobian $J(\mathcal{E}) = \mathbb{C}/\{n(2w_1)+m(2w_2) \mid n, m\in \mathbb{Z}\}$. Here $p_\infty$ is the point at infinity on the compactified curve $\mathcal{E}$. 

We start with an arbitrary point $z_0=2w_1c_1+2w_2c_2$ in the Jacobian, with some constants $c_1$ and $c_2$. Let the point $q_0$ on $\mathcal{E}$ be  the Jacobi inversion of $z_0$, that is
\begin{equation*}
	\mathcal{A}(q_0)=\int_{p_\infty}^{q_0}\frac{du}{v} = z_0= 2w_1c_1+2w_2c_2.
\end{equation*}
We denote by $q^*$ the elliptic involution of a point $q$ of $\mathcal{E}$. The following differential of the third kind:
\begin{equation}\label{eq.5.1.8}
	\Omega_1:=\Omega_{q_0,q_0^*} - 4 \pi{i}c_2\omega,
\end{equation}
is the main object of study in \cite{Dragovic-Shramchenko2019}. Here $\Omega_{q_0,q_0^*}$ is the normalized differential of the third kind with poles at $q_0$ and $q_0^*$ of residues $1$ and $-1$, respectively, and $\omega$ the normalized holomorphic differential $w=\frac{1}{2w_1}\frac{du}{v}$.

Differential $\Omega_1$ has two zeros paired by the elliptic involution. Denote by $y$ their projection on the $u$-sphere. This projection, as a function of $x$, satisfies a Painlev\'e-VI equation:

\begin{theorem}[Dragovi\'{c}-Shramchenko (2019) \cite{Dragovic-Shramchenko2019}, Theorem 1]\label{thm.5.1} 
If $x$ varies, while $c_1, c_2$ stay fixed, then $y(x)$, the position of zero of $\Omega_1$, is the Okamoto transformation \eqref{eq.5.1.6} of $y_0$ and satisfies Painlev\'e VI with constants \eqref{eq.5.1.2}
\end{theorem}

One can observe that the periods of the differential $\Omega_1$ are preserved under the deformation of Theorem \ref{thm.5.1}, since its $a$-period  is $-4\pi i c_2$ and its $b$-period is $4\pi i c_1$.

\subsection{From isoperiodic families to explicit algebraic solutions to Painlev\'{e} VI equations}\label{sec.5.2}
\subsubsection{$n=3$}\label{sec.5.2.1}
If we substitute eq. \eqref{eq.3.1.1} into the eq. \eqref{eq.2.0.10}, we obtain
\begin{eqnarray}\label{eq.5.2.1}
\det\left(\lambda\mathcal{D}(x_E,y_E) +\mathcal{P}(p)\right)&=&- \lambda^3 -\left(\lambda (p + x_E)  + p\right)^2.
\end{eqnarray}

We choose $E=(1,0)$. The cubic polynomial from eq. \eqref{eq.5.2.1} factors as
\begin{eqnarray}\label{eq.5.2.2}
\det\left(\lambda\mathcal{D}(1,0) +\mathcal{P}(p)\right)&=&- (\lambda + 1) \left(\lambda^2 +p(p+2)\lambda  + p^2\right).
\end{eqnarray}
The zeros of the cubic polynomial are:
\begin{equation*}
    \lambda_1(p)=-1, \quad \lambda_{\pm}(p)=\frac{1}{2}\left(-(p^2+2p)\pm \sqrt{p^3(p+4)}\right). 
\end{equation*}
Let us consider the following M\"{o}bius transformation:
\begin{equation}\label{eq.5.2.3}
    u=\phi_p(\lambda):=\frac{\lambda-\lambda_{-}}{\lambda_{+}-\lambda_{-}}.
\end{equation}
This gives
\begin{eqnarray}
    x(p)&:=&\phi_p(\lambda_1)=\frac{(p^2 + 2p - 2)+\sqrt{p^3(p+4)}}{2 \sqrt{p^3(p+4)}},\label{eq.5.2.4}\\
    y_0(p)&:=&\phi_p(0)=\frac{(p^2 + 2p)+\sqrt{p^3(p+4)}}{2 \sqrt{p^3(p+4)}}.\label{eq.5.2.5}
\end{eqnarray}
Since, by assumption $p\ne 0$, we are getting two solutions, one for $p>0$ and one for $p<0$:  $y_0(p)$ is a Picard solution to the Painlev\'e VI $(0,0,0, 1/2)$ as a function of $x(p)$.

\noindent
Using
\begin{eqnarray*}
   \frac{dy_0}{dx}=-\frac{p}{3},
\end{eqnarray*}
and the Okamoto transformation eq. \eqref{eq.5.1.6}, we obtain
\begin{eqnarray}\label{eq.5.2.6}
    y(p)=\frac{\left(p^2+2p+\sqrt{p^3(p+4)}\right)\left(-p^2-4p+3\sqrt{p^3(p+4)}\right)}{4p(p+1)\sqrt{p^3(p+4)}}.
\end{eqnarray}

\begin{theorem}\label{thm.5.2}
The functions $y_0(p)$ and $y(p)$ as functions of $x(p)$ given by the equations \eqref{eq.5.2.5} and \eqref{eq.5.2.6}, are solutions to the Painlev\'e VI $(0,0,0, 1/2)$  and the Painlev\'e VI $(1/8,-1/8,1/8,3/8)$, respectively. 
\end{theorem}

The proof follows from \cite{Dragovic-Shramchenko2019}, see Theorem \ref{thm.5.1} above.

\begin{remark}
Hitchin in \cite{Hitchin1995} obtained the following solution to the Painlev\'e VI $(1/8,-1/8,1/8,3/8)$ for $n=3$:
\begin{equation*}
    y=\frac{s^2(s+2)}{s^2+s+1},\quad x=\frac{s^3(s+2)}{2s+1}.
\end{equation*}
After eliminating $s$ from these equations, one obtains:
\begin{equation*}
    x y^3 (y+2) + x^3 (2 y-1) - x^2 y (y^3 - 2 y^2 + 6 y - 2) - y^4=0.
\end{equation*}
The solution, given by the eqs. \eqref{eq.5.2.4} and \eqref{eq.5.2.6}, also satisfies the above equation.
\end{remark}

\subsubsection{$n=4$}\label{sec.5.2.2}
From the eq. \eqref{eq.2.0.10}, we obtain
\begin{eqnarray}\label{eq.5.2.7}
\det\left(\lambda\mathcal{D}(0,0) +\mathcal{P}(p)\right)=-(\lambda + 1) (\lambda^2+\lambda p^2+p^2).
\end{eqnarray}
The zeros of the cubic polynomial from \eqref{eq.5.2.7} are:
\begin{equation}\label{eq.5.2.8}
    \lambda_1(p)=-1,\quad \lambda_{\pm}(p)=\frac{1}{2}\left(-p^2\pm \sqrt{p^4-4p^2}\right). 
\end{equation}
By a similar application of the  M\"{o}bius  transformation eq. \eqref{eq.5.2.3} for eq. \eqref{eq.5.2.5}, we obtain
\begin{eqnarray}
    x(p)&:=&\phi_p(\lambda_1)=\frac{p^2 - 2 + \sqrt{p^2(p^2 - 4)}}{2\sqrt{p^2(p^2 - 4)}},\label{eq.5.2.9}\\
    y_0(p)&:=&\phi_p(0)=\frac{1}{2}\left(1+\frac{p^2}{\sqrt{p^2(p^2 - 4)}}\right).\label{eq.5.2.10}
\end{eqnarray}

Using
\begin{eqnarray*}
  \frac{dy_0}{dx}=\frac{p^2}{2},
\end{eqnarray*}
and the Okamoto transformation eq. \eqref{eq.5.1.6}, we get
\begin{eqnarray}\label{eq.5.2.11}
    y(p)=\frac{p^2 + \sqrt{p^2(p^2 - 4)}}{2p^2}.
\end{eqnarray}

Applying \cite{Dragovic-Shramchenko2019}, see Theorem \ref{thm.5.1} above, we get:
\begin{theorem}\label{thm.5.3}
The functions $y_0(p)$ and $y(p)$ as functions of $x(p)$, given by equations \eqref{eq.5.2.10}, \eqref{eq.5.2.11} are solutions to the Painlev\'e VI $(0,0,0, 1/2)$ and the Painlev\'e VI $(1/8,-1/8, 1/8, 3/8)$, respectively. 
\end{theorem}

After eliminating $p$ from the eqs. \eqref{eq.5.2.9} and \eqref{eq.5.2.10}, we obtain
\begin{equation}\label{eq.5.2.12}
    y_0^2-2xy_0+x=0.
\end{equation}
Similarly, eliminating $p$ from the eqs. \eqref{eq.5.2.9} and \eqref{eq.5.2.11}, we again obtain
\begin{equation}\label{eq.5.2.13}
    y^2-2xy+x=0.
\end{equation}
In addition, $y_0$ and $y$ are related as follows:
\begin{equation*}
    y=\frac{y_0}{2y_0-1}.
\end{equation*}

\begin{remark}
    The solutions  $y_0(p)$ and $y(p)$ obtained in \cite{Dragovic-Radnovic2024} for $n=4$ (as well as in \cite{Hitchin1995})  as functions of $x(p)$ satisfy relations $y_0=-y$ and $y^2=x$. The solutions  $y_0(p)$ and $y(p)$ obtained  for $n=4$ in this paper satisfy \eqref{eq.5.2.13} and \eqref{eq.5.2.12}.
\end{remark}
For different approaches to algebraic solutions to Painlev\'e VI equations see e.g., \cite{Dragovic-Shramchenko2021,DragovicGontsovShramchenko2021} and references therein. An interesting question would be to find some other natural $n$-isoperiodic families of conics that would generate algebraic solutions of Painlev\'e VI equations, for each $n, m$, as listed in  \cite{Dragovic-Shramchenko2021}.

\thanks{\textbf{Acknowledgments}. This work is dedicated to the Memory of Academician A. A. Bolibrukh, on the occasion of his 75th anniversary. 

Authors would like to acknowledge the extensive use of computer algebra systems: \textit{Mathematica} and \textit{Geogebra} for various symbolic and numerical computations, and generating the figures in this research.}

This research has been partially supported by  the Science Fund of Serbia grant Integrability and Extremal Problems in Mechanics, Geometry and Combinatorics, MEGIC, Grant No. 7744592 and the Simons Foundation grant no. 854861.

\end{document}